\documentclass[amscd,amssymb,verbatim,10pt]{amsart}

\usepackage[fontsize = 10bp]{fontsize}

\usepackage{amsthm}
\usepackage{amssymb, latexsym, amsmath, amscd, array, graphicx}
\usepackage[linktocpage=true]{hyperref}
\usepackage[all]{xy}
\usepackage[scr]{rsfso}

\newtheorem{swtheorem}{Theorem}
\swapnumbers
\usepackage{verbatim}
\usepackage{mathtools}
\usepackage{color}
\hypersetup{
    colorlinks=true,
    linkcolor=red,
    urlcolor=black,
    citecolor=blue
           }
\usepackage{tikz-cd}
\usepackage{bbm,graphicx}

\theoremstyle{plain}

\newtheorem{thm}{Theorem}[section]
\newtheorem{theorem}[thm]{Theorem}

\newtheorem{lemma}[thm]{Lemma}
\newtheorem{conjec}[thm]{Conjecture}
\newtheorem{prop}[thm]{Proposition}
\newtheorem{cor}[thm]{Corollary}

\theoremstyle{definition}
\newtheorem{defn}[thm]{Definition}

\newtheorem{remark}[thm]{Remark}

\newtheorem{ex}[thm]{Example}

\newtheorem{question}[thm]{Question}

\DeclareMathOperator{\cat}{\mathsf{cat}}
\DeclareMathOperator{\dcat}{\mathsf{dcat}}

\DeclareMathOperator{\s}{Susp}

\DeclareMathOperator{\cd}{{\rm cd}}

\DeclareMathOperator{\supp}{{\rm supp}}

\def\B{{\mathcal B}}

\def\A{{\mathcal A}}
\def\M{{\mathcal M}}
\def\P{{\mathcal P}}
\def\H{{\mathcal H}}

\def\F{{\mathcal F}}
\def\O{{\mathcal O}}

\usepackage{lipsum} 

\newcommand \pa[2]{\frac{\partial #1}{\partial #2}}

\def\scr{\mathcal}
\def\A{{\scr A}}
\def\B{{\scr B}}
\def\O{{\scr O}}
\def\C{{\mathbb C}}
\def\Z{{\mathbb Z}}
\def\Q{{\mathbb Q}}
\def\R{{\mathbb R}}
\def\N{{\mathbb N}}

\def\1{\hbox{\rm\rlap {1}\hskip.03in{\rom I}}}
\def\Bbbone{{\rm1\mathchoice{\kern-0.25em}{\kern-0.25em}
{\kern-0.2em}{\kern-0.2em}I}}

\def\pa{\partial}

\def\wt{\widetilde}
\def\wh{\widehat}

\def\ov{\overline}

\long\def\forget#1\forgotten{}

\newcommand\ver[1]{\marginpar{\tiny Changed in Ver \VER}}

\date{\today}

\begin{document}

\title[Distributional category of manifolds]{Distributional category of manifolds}

\author[Ekansh Jauhari]{Ekansh Jauhari}

\address{Ekansh Jauhari, Department of Mathematics, University
of Florida, 358 Little Hall, Gainesville, FL 32611-8105, USA.}
\email{ekanshjauhari@ufl.edu}



\begin{abstract}
Recently, a new homotopy invariant of metric spaces, called the distributional category, was defined, which provides a lower bound to the Lusternik--Schnirelmann (LS) category. In this paper, we obtain several sufficient conditions for the distributional category ($\dcat$) of a closed manifold to be maximum, i.e., equal to its classical LS-category ($\cat$). These give us many new computations of $\dcat$, especially for some essential manifolds and (generalized) connected sums. In the process, we also determine the $\dcat$ of closed $3$-manifolds having torsion-free fundamental groups and various closed geometrically decomposable $4$-manifolds. Finally, we extend some of our results to closed Alexandrov spaces with curvature bounded below and discuss their $\cat$ and $\dcat$ in dimension $3$. 
\end{abstract}




\maketitle

\section{Introduction}\label{Introduction}
Given a continuous map $f:X \to Y$, the \emph{Lusternik--Schnirelmann (LS) category of} $f$, denoted $\cat(f)$, is the smallest integer $n$ such that $X$ can be covered by $n+1$ number of open sets $U_i$ for each of which the restriction $f_{|U_i}$ is null-homotopic, see~\cite{BG},~\cite{Ja}. When $f$ is the identity map, we recover the classical notion of the \emph{LS-category of a space} $X$, denoted $\cat(X)$, see~\cite{Ja},~\cite{CLOT}.

The LS-category of closed manifolds has been well-studied,~\cite{GLGA},~\cite{RO},~\cite{R1}, \cite{OR},~\cite{CLOT},~\cite{KR},~\cite{DKR},~\cite{OS},~\cite{DS}. Gómez-Larrañaga and González-Acuña completely described in~\cite{GLGA} the LS-category of closed $3$-manifolds in terms of their fundamental groups. This was supplemented by the work of Oprea and Rudyak who used a different method to obtain the same description in~\cite{OR}. For the connected sum of two closed $n$-manifolds $M$ and $N$ for any $n \ge 1$, Dranishnikov and Sadykov proved in~\cite{DS} the formula $\cat(M\#N) = \max\{\cat(M),\cat(N)\}$. 

Recently, in a joint work with Dranishnikov~\cite{DJ}, we defined a new homotopy invariant of spaces, called the distributional category, denoted $\dcat(X)$ for a space $X$ (see Section~\ref{lastnew} for its motivation). In general, $\dcat(X) \le \cat(X)$. The invariant $\dcat$ satisfies most of the nice properties as $\cat$. However, $\dcat$ is a different notion than $\cat$ and has various contrasting properties as well (see Section~\ref{pre1} for the details). Unlike the case of the classical LS-category, the distributional category is known only for a few classes of spaces~\cite{DJ},~\cite{KW},~\cite{Dr2}, and there are a limited number of tools available for its computation.

Since the inception of this nascent invariant whose research is an emerging topic, the following tantalizing question has been around.

\begin{question}\label{myques}
For what spaces $X$ do we have the equality $\dcat(X)=\cat(X)$?
\end{question}

To answer the above question, it is worthwhile to check if the techniques that are typically useful in determining LS-category remain somewhat useful in determining distributional category as well. One such technique is lower bounding the categories of a finite CW complex with its rational cup-length,~\cite{DJ}. Another example is the Eilenberg--Ganea technique~\cite{EG} of expressing the categories of the classifying space of a torsion-free discrete group in terms of the cohomological dimension of the group,~\cite{KW}. 

As a result of these techniques, some spaces on which the distributional category agrees with the LS-category include closed surfaces (other than $\R P^2$), complex projective spaces, finite products of spheres, and classifying spaces of torsion-free discrete groups. Spaces on which the two notions of category don't agree include real projective spaces and classifying spaces of finite groups,~\cite{DJ},~\cite{KW}. These examples make Question~\ref{myques} very interesting. 

The main objective of this paper is to answer Question~\ref{myques} by finding general classes of compact metric spaces (in particular, closed manifolds) on which the two notions of category coincide. We do this by presenting various sufficient conditions for the distributional category of a closed manifold to be equal to its classical LS-category. We show that careful use of some obstruction-theoretic methods turns out to be very helpful in computing the distributional category in various cases. 

As a result of our sufficient conditions, we obtain analogs of some classical results on the LS-category of closed manifolds (especially the two results mentioned in Paragraph $2$ above) for the distributional category. Our techniques also enable us to partially extend some of these sufficient conditions to a more general class of closed Alexandrov spaces with curvature bounded below,~\cite{BGP},~\cite{BBI}. This helps us determine the classical LS-category of closed Alexandrov spaces in some cases. 

Let us describe our main results and methods more precisely.

For closed essential manifolds with torsion-free fundamental groups, we define the notion of \emph{cap property} in Section~\ref{torsionfree}. Using this notion and a characterization of distributional category from Section~\ref{pre1}, we obtain the following.

\begin{swtheorem}\label{1}
    Let $M$ be a closed essential $n$-manifold such that $\pi_1(M)$ is torsion-free. If $M$ satisfies the cap property, then $\dcat(M)=\cat(M)=n$.
\end{swtheorem}

To prove Theorem~\ref{1} and other major results, we use some primary obstructions (see details in Section~\ref{torsionfree}) that turn out to be quite useful for our purpose due to the works of Knudsen and Weinberger~\cite{KW} and Dranishnikov~\cite{Dr2}. 

Using the relationship between the orientation sheaves of a connected sum and its summands, we generalize Theorem~\ref{1} to the following.

\begin{swtheorem}\label{2}
    Let $M$ be a closed $n$-manifold and $N$ be a closed essential $n$-manifold such that $\pi_1(N)$ is torsion-free. If $N$ satisfies the cap property, then we have the equality $\dcat(M \# N) = \cat(M\# N) = n$. 
\end{swtheorem}

As an application of Theorem~\ref{2}, we completely determine the distributional category of closed $3$-manifolds having torsion-free fundamental groups, thereby obtaining an analog of the main result of~\cite{GLGA} and~\cite{OR} in all but one case.

\begin{swtheorem}\label{bb}
    If $M$ is a closed $3$-manifold such that $\pi_1(M)$ is torsion-free, then $\dcat(M) = \cat(M)$. More precisely,
    \[
\dcat(M) =\begin{cases}
    1 & \text{if } \pi_1(M) = 0 \\
    2 & \text{if } \pi_1(M)\ne 0 \text{ is free} \\
    3 & \text{if } \pi_1(M) \text{ is not free but torsion-free}
\end{cases}
\]
\end{swtheorem}

Next, we turn towards some generalized connected sums (or fiber sums) of orientable manifolds. In Section~\ref{6.1}, we define the notion of the connected sum $M\#_L N$ of two closed smooth orientable $n$-manifolds $M$ and $N$ \emph{joined} along a space $L$ which is obtained using the embeddings of a common closed submanifold $S$ in $M$ and $N$. For these generalized connected sums, we get the following.

\begin{swtheorem}\label{3}
    Let $M$ and $N$ be closed smooth orientable $n$-manifolds glued along $L$ (in the sense of Section~\ref{6.1}) to form $M\#_L\hspace{0.2mm}N$. Let $N/L$ be a closed orientable $n$-psuedomanifold. If $N/L$ is aspherical, then $\dcat(M\#_L\hspace{0.2mm}N) = \cat(M\#_L\hspace{0.2mm}N) = n$. 
\end{swtheorem}

We also look at some other generalized connected sums, namely the doubles of punctured compact manifolds with boundaries. For any $k\ge 1$, let us call the process of the removal of $k$ number of small, disjoint open $n$-discs (or the interiors of closed $n$-discs) from a closed orientable $n$-manifold $M$ a \emph{$k$-puncturing} of $M$. For the doubles of the finite products of these manifolds, we obtain the following.

\begin{swtheorem}\label{4}
    For $1\le i \le n$, let $k_i$, $r_i\ge 1$ and let $M_i$ be a closed orientable aspherical $r_i$-manifold. Let $X_i$ be obtained from $M_i$ after a $k_i$-puncturing of $M_i$. If $M = \prod_{i=1}^n X_i$, then $\dcat(\mathcal{D}M) = \sum_{i=1}^n r_i\hspace{0.5mm}$ for the double $\mathcal{D}M=M\#_{\pa M}\ov M$ of $M$. 
\end{swtheorem}

Finally, we turn towards the class of closed Alexandrov spaces with curvature bounded below. In Section~\ref{definitions}, we define the notion of the connected sum of closed Alexandrov $n$-spaces along their regular points as a direct extension of the corresponding notion from~\cite{decomposition} for the case $n=3$. Using techniques similar to those used in the proof of Theorem~\ref{2}, we obtain the following.

\begin{swtheorem}\label{5}
    Let $M$ and $N$ be closed Alexandrov $n$-spaces. If $N$ is an aspherical manifold, then $\dcat(M\#N) = \cat(M\#N)=n$ for the connected sum $M\# N$ along regular points of $M$ and $N$.
\end{swtheorem}

By relaxing the manifold condition on $N$ in Theorem~\ref{5} and instead imposing orientability on the spaces involved, we get the following.

\begin{swtheorem}\label{6}
    Let $M$ and $N$ be closed orientable Alexandrov $n$-spaces. If $N$ is aspherical, then $\dcat(M\# N) =\cat(M\# N)=n$ for the connected sum $M\# N$ along regular points of $M$ and $N$.
\end{swtheorem}

To the best of our knowledge, the classical LS-category of closed Alexandrov $n$-spaces has not been studied in the literature. In this paper, we initiate a study in that direction by proving Theorems~\ref{5} and~\ref{6} using some novel ideas. 

We note that each of the above statements is proven in dimensions $\ge 3$. A closed $1$-manifold is a circle and a closed $2$-manifold is a surface of genus $\ge 0$. Also, an Alexandrov $n$-space is a manifold for $n \le 2$ (see Section~\ref{definitions}). The distributional category of these spaces has been completely determined in~\cite[Section 6.1]{DJ}. In fact, $\dcat(M) = \cat(M)$ if $M$ is a closed $n$-manifold for $n \le 2$ other than $\R P^2$ and $\dcat(\R P^2) = 1$. So, in this paper, we shall always work in dimensions $n \ge 3$ unless explicitly stated otherwise.

\subsection*{Significance of our results} 
Here, we highlight the relevance of each of the above theorems for $\dcat$ and $\cat$.

Theorem~\ref{1} generalizes the fact from~\cite{KW} and~\cite{Dr2} that $\dcat(M) = \dim(M)$ for any closed aspherical $n$-manifold $M$. Theorem~\ref{2} gives a partial connected sum formula for $\dcat$ analogous to the formula of~\cite[Theorem 1]{DS} for $\cat$. Theorems~\ref{3} and~\ref{4} determine $\dcat$ of some orientable generalized connected sums that are not covered in Theorem~\ref{2}. Theorems~\ref{5} and~\ref{6} extend Theorem~\ref{2} to some non-manifold cases and provide a partial connected sum formula for $\cat$ and $\dcat$ of (possibly non-manifold) closed Alexandrov spaces. This is also relevant because the connected sum formula for $\cat$ from~\cite{DS} is not known for closed Alexandrov spaces.

\subsection*{Applications of our results} 
Here, we highlight the new computations of $\dcat$ that are obtained as a consequence of the above theorems.

Theorem~\ref{1} computes $\dcat$ of several symplectically aspherical manifolds that are not aspherical. Theorem~\ref{2} computes $\dcat$ of connected sums involving the examples considered in Theorem~\ref{1} and aspherical manifolds whose rational cup-lengths are not maximum. It also helps determine $\dcat$ of closed $3$-manifolds, see Theorem~\ref{bb}. Theorem~\ref{bb} proves an analog of Rudyak's conjecture~\cite{R1},~\cite{R2} for $\dcat$ of low-dimensional manifolds in various cases. Theorem~\ref{4} computes $\dcat$ and $\cat$ of some closed geometrically decomposable $4$-manifolds~\cite{Hi} that are not aspherical. The computation of the classical LS-category of these spaces seems missing from the literature, so Theorem~\ref{4} indeed produces some new computations. Theorem~\ref{5} calculates $\dcat$ of connected sums of non-manifolds with some examples considered in Theorem~\ref{2}. Theorems~\ref{5} and~\ref{6} also help in estimating $\dcat$ and $\cat$ of closed Alexandrov $3$-spaces in some special cases. 

Lastly, we note that distributional category ($\dcat$) is a fairly new numerical invariant; indeed, many aspects of its behavior remain to be understood fully. We hope that the results of this paper, which suggest a close relationship of $\dcat$ with other classical invariants and compute $\dcat$ of various manifolds using some novel techniques, will stimulate interest in this invariant and further research activity.


\subsection*{Structure of the paper}
In Section~\ref{Preliminaries}, we revisit the theory of distributional category and recall some facts and notions from algebraic topology. In Section~\ref{essential}, we prove Theorem~\ref{1} and see its various applications. In Section~\ref{connsums}, we prove Theorem~\ref{2} and use it to obtain many new computations of $\dcat$. In Section~\ref{lowdim}, we study $\dcat$ of low-dimensional manifolds, prove Theorem~\ref{bb}, and apply it to an analog of Rudyak's conjecture for distributional category. Section~\ref{gensum} contains the proof of Theorems~\ref{3} and~\ref{4}, and the application of Theorem~\ref{4} to closed geometrically decomposable $4$-manifolds. In Section~\ref{alexalex1}, we study closed Alexandrov spaces, prove Theorems~\ref{5} and~\ref{6}, and provide some new computations for both $\cat$ and $\dcat$. We end this paper in Section~\ref{alexalex2}, where we focus on non-manifold Alexandrov $3$-spaces, discuss some problems in determining their $\cat$ and $\dcat$, and describe spaces $X$ for which $\cat(X) > 1$ using their curvature and cohomogeneity. 

\subsection*{Notations and conventions} All the topological spaces considered in this paper are compact and path-connected metric spaces. We use the symbol ``$=$" to denote homeomorphisms of spaces and isomorphisms of groups, and ``$\simeq$" to denote homotopy equivalences. Only the rings with unity are considered and the tensor product is taken over $\Z$. All the results and computations in this paper are obtained for dimensions $\ge 3$.

\section{Preliminaries}\label{Preliminaries}

\subsection{Motivation for $\dcat$}\label{lastnew}
Given a pointed topological space $(Z,z_0)$, let us define $P(Z) = \left\{\phi: [0,1] \to Z \mid \phi(1) = z_{0}\right\}$, and let $p_Z: P(Z) \to Z$ be the evaluation fibration defined as $p(\phi)= \phi(0)$ for all $\phi\in P(Z)$. For each $n\ge 1$, we define the $n$-th \emph{Ganea space}, denoted $G_{n}(Z)$, to be the fiberwise iterated join of $n$ copies of $P(Z)$ along $p_Z$, i.e.,
\[
G_{n}(Z) = \left\{\sum_{i=1}^{n} \lambda_{i} \phi_{i} \hspace{1mm} \middle| \hspace{1mm} \phi_{i} \in P(Z), \sum_{i=1}^{n}\lambda_{i} = 1, \lambda_{i} \geq 0, \phi_{i}(0) = \phi_{j}(0)\right\}.
\]
Each element of $G_n(Z)$ is a formal ordered linear combination where the terms with weights $\lambda_i=0$ are dropped.
On this space, we define the $n$-th \emph{Ganea fibration} $p_{n}^{Z} : G_{n}(Z) \to Z$ such that $p_{n}^{Z} (\sum_{i=1}^{n} \lambda_{i} \phi_{i}) = \phi_{i}(0)$ for any $i\le n$ such that $\lambda_{i} > 0$.

The following theorem gives the Ganea--Schwarz characterization of the LS-category of pointed maps,~\cite{Sch},~\cite{BG},~\cite{Ja},~\cite{CLOT}.

\begin{theorem}
    Given a pointed map $f:(X,x_0)\to (Y,y_0)$,  $\cat(f) < n$ if and only if there exists a lift of $f$ with respect to the fibration $p_{n}^{Y}: G_{n}(Y) \to Y$.
\end{theorem}

So, $\cat(f)$ is the smallest integer $n$ such that each $x\in X$ is mapped to a formal \emph{ordered} linear combination of at most $n+1$ non-negatively weighted paths in $P(Y)$ where each path has initial point $f(x)$ and endpoint $y_0$.

One of the main motivations for defining the distributional version of this classical notion of LS-category is to think about an invariant that assigns to a given pointed map $f:(X,x_0)\to (Y,y_0)$ the smallest integer $n$ such that each $x\in X$ is mapped to a formal \emph{unordered} linear combination of at most $n+1$ weighted paths in $P(Y)$ having initial point $f(x)$ and endpoint $y_0$ (in~\cite{DJ}, we considered the special case when $f$ is the identity map). Then, as in~\cite{DJ}, one can develop a theory of this new invariant parallel to the theory of the classical invariant $\cat$.

\subsection{Distributional category}\label{pre1}For a metric space $Z$, let $\mathcal{B}(Z)$ denote the set of probability measures $\mu=\sum\lambda_i\delta_{x_i}$ on $Z$, where $\sum\lambda_i=1$, $\lambda_i\ge 0$, and $\delta_{x_i}$ is the point Dirac measure at $x_i\in Z$. For any $n \ge 1$, let
\[
\mathcal{B}_{n}(Z) = \left\{\mu \in \mathcal{B}(Z) \mid |\supp(\mu)| \le n \right\}
\]
denote the space of probability measures on $Z$ supported by at most $n$ points, equipped with the Lévy--Prokhorov metric~\cite{Pr} (or the Wasserstein distance metric), see also~\cite[Section 3.1]{DJ}. Here, $\supp(\mu)$ denotes the support of $\mu$. Given the pointed metric space $(Z,z_0)$, for each $z\in Z$, let $P(z,z_0)=\{\phi\in P(Z)\mid \phi(0)=z\}$.

\begin{defn}
    The {\em distributional category} of a map $f:(X,x_0) \to (Y,y_0)$, denoted $\dcat(f)$, is the smallest integer $n$ for which there exists a continuous map $H: X \to \B_{n+1}(P(Y))$ such that $H(x) \in \B_{n+1}(P(f(x),y_0))$ for all $x \in X$.
\end{defn}

Upon taking $f$ as the identity map, we get the notion of the distributional category of a space $X$, denoted $\dcat(X)$. For the following discussion, we refer to~\cite[Sections 3.2 and 4.2]{DJ}. 

Given a space $X$, $\dcat(X)$ is a lower bound to its classical LS-category $\cat(X)$. It is easy to see that $\dcat(X) = 0$ if and only if $X$ is contractible. Hence, $\dcat(S^n) = 1$ for all $n$. Interestingly, $\dcat(\R P^n) = 1$ for all $n$, in contrast to the well-known equality $\cat(\R P^n) = n$. This result leads to various differences between the theories of $\cat$ and $\dcat$, see~\cite[Section 6.1]{DJ} and~\cite[Remark 4.13]{J2}. In general, $\dcat(X)$ can be arbitrarily large; indeed, the rational and the real cup-lengths of $X$
are lower bounds to $\dcat(X)$. However, we note that, unlike the case of the classical LS-category, the $R$-cup-length of $X$, denoted $c\ell_{R}(X)$, is not necessarily a lower bound to $\dcat(X)$ for a general ring $R$ due to the example of $\R P^n$. Furthermore, $\dcat$ is a homotopy invariant of spaces. So, for a finitely generated discrete group $\Gamma$, one can define $\dcat(\Gamma): = \dcat(B\Gamma)$, where $B\Gamma$ denotes the unique (up to homotopy) classifying space of $\Gamma$. 

The well-known result of Eilenberg and Ganea~\cite{EG} says that $\cat(\Gamma) = \cd(\Gamma)$, where $\cd(\Gamma)=\max\{k\in\mathbb{Z}\mid H^k(\Gamma,\A)\ne 0 \text{ for some } \Z\Gamma\text{-module } \A\}$ is the cohomological dimension of $\Gamma$. For any finite group $\Gamma$, while $\cat(\Gamma)=\cd(\Gamma)$ is infinite, it follows from~\cite[Theorem 7.2]{KW} that $\dcat(\Gamma)\le |\Gamma|-1$ (see also~\cite[Section 8.A]{J1}). In fact, for any prime $p\ge 2$, $\dcat(\Z_p)=p-1$ due to~\cite[Theorem 5.12]{Dr2} (see also~\cite{DJ} for the case $p=2$). Therefore, an analog of the Eilenberg--Ganea theorem does not hold in the theory of distributional category for groups having torsion. However, such an analog holds for torsion-free discrete groups.

\begin{theorem}[\protect{\cite[Theorem 7.4]{KW},~\cite[Theorem 5.3]{Dr2}}]\label{eg}
    If $\Gamma$ is a torsion-free discrete group, then $\dcat(\Gamma) = \textup{acat}(\Gamma) = \cat(\Gamma) = \cd(\Gamma)$.
\end{theorem}

Here, $\text{acat}$ denotes the ``analog category", which is a new homotopy invariant of topological spaces introduced recently by Knudsen and Weinberger in~\cite{KW}, who first proved that $\text{acat}(\Gamma) = \cd(\Gamma)$. We note that the notions of $\text{acat}$ and $\dcat$ are very similar, see~\cite[Conjecture 1.2]{KW},~\cite[Section 8.A]{J1}, and~\cite[Page 2]{Dr2}. 

Recall the evaluation fibration $p_Z: P(Z) \to Z$  defined as $p_Z(\phi) = \phi(0)$ for all $\phi\in P(Z)$. For each $n \ge 1$, the result of the fiberwise application of the functor $\mathcal B_n$ to the based path space $P(Z)$ yields a space
\[
P(Z)_n = \bigcup_{z \in Z}\mathcal{B}_n  (p_Z^{-1}(z)) =  \left\{ \mu \in \B_n(P(Z)) \hspace{1mm} \middle| \hspace{1mm} \supp(\mu) \subset p_Z^{-1}(z), \hspace{1mm} z \in Z \right\}
\]
and a continuous map $\mathcal{B}_n(p_Z) : P(Z)_n \to Z$ such that $\mathcal{B}_n(p_Z)(\mu) = z$ whenever $\mu \in \mathcal{B}_n (p_Z^{-1}(z))$. We refer to~\cite[Section 5]{DJ} for a proof of the following.

\begin{prop}\label{characterization}
The map $\mathcal{B}_n(p_Z) : P(Z)_n \to Z$ is a Hurewicz fibration, and $\dcat(Z) < n$ if and only if the fibration $\mathcal{B}_n(p_Z)$ admits a section.
\end{prop}

We note that the distributional category of a space is a special case of the general notion of the ``distributional sectional category" of a fibration, denoted $\mathsf{dsecat}$, which is the distributional analog of the classical notion of the sectional category~\cite{Sch},~\cite{Ja}. We refer the reader to~\cite[Section 5]{J1} for details on $\mathsf{dsecat}$.

Given a map $f: (X,x_0) \to (Y,y_0)$, in analogy with the nice properties of $\cat(f)$ discussed in~\cite{BG} and~\cite[Section 7]{Ja}, most of the statements for $\dcat(X)$ obtained in~\cite{DJ} can be easily generalized to get the corresponding statements for $\dcat(f)$. 

\begin{prop}\label{obvious}
\begin{enumerate}
    \itemsep-0.2em 
    \item  $\dcat$ is a homotopy invariant of pointed maps.
    \item $\dcat(f) \le \min\{\cat(f),\dcat(X),\dcat(Y)\}$.\label{t2}
    \item For any map $g:(Y,y_0) \to (Z,z_0)$, $\dcat(g \circ f) \le \min\{\dcat(f),\dcat(g)\}$.
    \item Let $\beta_i \in H^{k_i}(SP^{n!}(Y);R)$, $1 \le i \le n$, for some ring $R$ and $k_i \ge 1$. If $\alpha_i = (\delta_n\circ f)^*(\beta_i)$ such that $\alpha_1 \smile \cdots \smile \alpha_n \ne 0$, then $\dcat(f) \ge n$.
\end{enumerate}
\end{prop}

Here, $\delta_n:Y \to SP^{n!}(Y)$ is the diagonal embedding into the $n!$ symmetric power of $Y$, which is defined as the orbit space of the natural action of the symmetric group $S_{n!}$ on the product space $Y^{n!}$. 

The following characterization was stated in~\cite[Proposition 2.8]{Dr2} as an extension of the second part of Proposition~\ref{characterization}. Here, we provide a simple proof.

\begin{prop}\label{lift}
    Given a pointed map $f:(X,x_0)\to (Y,y_0)$, $\dcat(f) < n$ if and only if there exists a lift of $f$ with respect to the fibration $\B_n(p_Y):P(Y)_n \to Y$.
\end{prop}

\begin{proof}
    If $\dcat(f) < n$, then for a basepoint $y_0 \in Y$, there exists a continuous map $H: X \to \B_n(P(Y))$ such that $H(x) \in \B_n(P(f(x),y_0))$ for each $x\in X$. In particular, the image of $H$ is in $P(Y)_n$. By definition of $H$, $\B_n(p_Y)(H(x)) = f(x)$ for all $x$. Hence, $H$ is the required lift. Conversely, if $g: X \to P(Y)_n$ is a lift of $f$ with respect to $\B_n(p_Y)$, then $\B_n(p_Y)(g(x)) = f(x)$ and thus, we must have $g(x) \in \B_n(P(f(x),y_0))$ for all $x \in X$. So, by definition, $\dcat(f) \le n-1$. 
\end{proof}

The following covering map inequality for $\dcat$ (which is a direct analog of the corresponding inequality for $\cat$, see~\cite[Corollary 1.45]{CLOT}) will be very useful in Sections~\ref{threemani} and~\ref{simpobs}.

\begin{theorem}[\protect{\cite[Theorem 3.7]{DJ}}]\label{covv}
    If $p:X\to Y$ is a covering map, then $\dcat(X)\le\dcat(Y)$.
\end{theorem}

\subsection{Topological basics}\label{basics} Given a discrete group $\Gamma$, its \emph{Berstein--Schwarz class}, denoted $\beta_\Gamma$, is the first obstruction to a lift of the classifying space $B\Gamma$ to its universal cover $E\Gamma$. If $I(\Gamma)$ denotes the augmentation ideal of the group ring $\Z\Gamma$, then $\beta_\Gamma\in H^1(B\Gamma;I(\Gamma))$. We refer the reader to~\cite[Section 2]{DKR} for a nice review of cohomology with local coefficients (see also~\cite[Section 3.H]{Ha}).

The Berstein--Schwarz class is ``universal" in some sense, see~\cite{Sch} and~\cite{DR}. The following well-known result, called the Berstein--Schwarz theorem, relates the cohomological dimension of a discrete group with its Berstein--Schwarz class.

\begin{theorem}[\protect{\cite{Sch},~\cite{DR}}]\label{bs}
    For a discrete group $\Gamma$, $\cd(\Gamma) = \max\{k \in \Z\mid \beta_\Gamma^k\ne 0\}$.
\end{theorem}

Given a fibration $f:E \to B$, where $B$ is a  CW complex and the fiber $F$ is $(n-2)$-connected for some $n\ge 3$, the \emph{primary obstruction} to a section of $f$ is a cohomology class $\kappa_n \in H^n(B;H_{n-1}(F))$, where coefficients are in the $\pi_1(B)$-module $H_{n-1}(F)$. 

For a discrete group $\Gamma$, it follows from~\cite{Sch} that the $n$-th power of its Berstein--Schwarz class, $\beta_\Gamma^n$, is the primary obstruction to a section of the $n$-th Ganea fibration $p_\Gamma:G_n(B\Gamma)\to B\Gamma$ defined in Section~\ref{lastnew}.

In our setting, for the fibration $\B_n(p_X):P(X)_n\to X$, the fiber is $\B_n(\Omega X)$ for the loop space $\Omega X$ and it is $(n-2)$-connected due to~\cite[Theorem 3.1]{Dr2} whenever $X$ is a CW complex. Hence, the fibration $\B_n(p_X)$ admits a section on the $(n-1)$-skeleton of $X$. So, the primary obstruction to a section of $\B_n(p_X)$, say $\kappa_n$, is a cohomology class of dimension $n$, i.e., $\kappa_n \in H^n(X;\F)$ for a $\pi_1(X)$-module $\F$. If $\dim(X)=n$, then $\kappa_n$ is the only obstruction to a section of $\B_n(p_X)$. Thus, in that case, $\kappa_n = 0$ if and only if $\B_n(p_X)$ admits a section. 

Moving forward, we reserve the notation $\kappa_n$ for such primary obstructions.

We end this section by recalling Poincaré duality. Given a closed manifold $M$ of dimension $n$, the orientation sheaf $\O_M$ on $M$ defines a fundamental class $[M]$ of $M$ such that for any local coefficients $\A$ on $M$ and integer $k \in \{0,\ldots,n\}$, the $k$-th homomorphism
\[
\text{PD}_{k}: H^k(M;\A) \to H_{n-k}(M;\A \otimes \O_M),
\]
that sends each $\alpha \in H^k(M;\A)$ to the cap product $[M]\frown \alpha \in H_{n-k}(M;\A \otimes \O_M)$, is an isomorphism, see~\cite[Chapter V, Sections 9 and 10]{Bre}. If a mapping $f:M \to N$ takes the orientation sheaf $\O_N$ to $\O_M$, then the \emph{degree} of $f$, denoted $\text{deg}(f)$, is the integer for which $f_*([M]) = \text{deg}(f)[N]$. If $\text{deg}(f) = \pm 1$, then $f$ induces an epimorphism $f_{\bullet}:\pi_1(M) \to \pi_1(N)$ of the fundamental groups.

\section{For essential manifolds}\label{essential}

Let $M$ be a closed $n$-manifold and $\pi: = \pi_1(M)$. We follow Gromov~\cite{Gr} and say that $M$ is \emph{essential} if there exists a classifying map $f:M \to B\pi$ that induces an isomorphism of the fundamental groups such that $f$ cannot be deformed into the $(n-1)$-skeleton of $B\pi$ (see also~\cite[Definition 3.1]{BD}). Some obvious examples of essential manifolds include aspherical manifolds. Non-aspherical examples include $\R P^n$ for each $n\ge 2$.

There are several equivalent definitions of essential manifolds. 

\begin{theorem}[\protect{\cite{KR}}]\label{iff}
    For a closed $n$-manifold $M$, the following are equivalent. 
        \vspace{-2mm}
\begin{enumerate}
    \itemsep-0.25em 
        \item $M$ is essential.\label{oneoone}
        \item $\cat(M) = \dim(M) = n$.\label{two}
        \item There exists a coefficient system $\A$ and a map $f:M\to B\pi$ such that the induced homomorphism $f^*:H^n(B\pi;\A) \to H^n(M;f^*(\A))$ is non-trivial.\label{333}
    \end{enumerate}
\end{theorem}

If $\A=\Q$ in part~(\ref{333}) above, then $M$ is called \emph{rationally essential}.

Of course, if $\dcat(M) = \dim(M)=n$, then $M$ is essential in view of the implication~(\ref{two})$\implies$(\ref{oneoone}) above. However, we note that an analog of Theorem~\ref{iff} does not hold for $\dcat$: even though $\R P^n$ is essential for each $n$, we have $\dcat(\R P^n) = 1 < n$ for all $n \ge 2$.

So, it is natural to look for essential manifolds whose distributional category is equal to the dimension of the manifold. 

\subsection{Torsion-free fundamental groups}\label{torsionfree}
We now restrict our attention to closed essential manifolds having torsion-free fundamental groups. Let $M$ be an essential $n$-manifold such that $\pi:=\pi_1(M)$ is torsion-free, and $f:M \to B\pi$ be a classifying map as above that cannot be deformed into the $(n-1)$-skeleton of $B\pi$. Since $M$ is essential, $(f^*(\beta_\pi))^n \ne 0$ for the Berstein--Schwarz class $\beta_\pi \in H^1(B\pi;I(\pi))$, see~\cite[Theorem 2.51]{CLOT} and also~\cite[Theorem 4.1]{KR}. This means $\beta_\pi^n \ne 0$ and thus, by Theorem~\ref{bs}, $\cd(\pi) \ge n$. Due to Theorem~\ref{eg}, we get $\dcat(B\pi)\ge n$. Hence, by Proposition~\ref{characterization}, the fibration $\B_n(p_\pi):P(B\pi)_n\to B\pi$ does not admit a section. Therefore, if $\dim(B\pi)=n$, then the primary obstruction $\kappa_n\in H^n(B\pi;\F)$ to a section of $\B_n(p_\pi)$ is non-zero.

\begin{defn}\label{capp}
        We say that $M$ satisfies the \emph{cap property} if $f_*([M]) \frown \kappa_n \ne 0$.
\end{defn}

\begin{remark}\label{formal}
    Let us formalize the cap property. The orientation sheaf $\O_M$ on $M$ is uniquely determined by the kernel of the homomorphism $w:\pi\to\Z_2=\text{Aut}(\Z)$, where $w$ is the first Stiefel--Whitney class of $M$, see~\cite[Page 32]{Ba}. Hence, for $M$ to have (co)homology with coefficients in $\O_M$ is the same as having (co)homology with coefficients in the twisted integers $\Z_w$. Since $[M]\in H_n(M;\O_M)=H_n(M;\Z_w)$ and $f:M\to B\pi$ induces an isomorphism of the fundamental groups, we have that $f_*([M])\in H_n(B\pi;\Z_w)$. Therefore, $f_*([M])\frown\kappa_n \in H_0(B\pi;\Z_w\otimes\F)$, where $\Z_w\otimes\F$ is a $\Z\pi$-module.
\end{remark}

\begin{proof}[Proof of Theorem~\ref{1}]
    Let $\kappa_n\in H^n(B\pi;\F)$ be the primary obstruction to a section of the fibration $\B_n(p_\pi):P(B\pi)_n\to B\pi$ and let $f:M\to B\pi$ be a classifying map that cannot be deformed into the $(n-1)$-skeleton of $B\pi$. Then the image $f^*(\kappa_n) \in H^n(M;f^*(\F))$ is the primary obstruction to a lift of $f$  with respect to $\B_n(p_\pi)$. Since $\dim(M)=n$, the only obstruction to such a lift of $f$ is $f^*(\kappa_n)$. Because $f$ induces an isomorphism (and hence an epimorphism) of fundamental groups, the induced map
    \[
    f_*:H_0(M;f^*(\A)) \to H_0(B\pi;\A)
    \]
    is an isomorphism for each coefficient system $\A$. Therefore, due to the cap property,
    \[
    f_*\left([M]\frown f^*(\kappa_n) \right) = f_*([M]) \frown \kappa_n \ne 0
    \]
    in $\Z_w\otimes\F$ coefficients (see Remark~\ref{formal}). Thus, $f^*(\kappa_n) \ne 0$. So, $f$ does not admit a lift with respect to $\B_n(p_\pi)$. Hence, by Proposition~\ref{lift}, we get $\dcat(f)\ge n$. This gives
    \[
    n \le \dcat(f) \le \dcat(M) \le \cat(M) = \dim(M) = n
    \]
    due to Proposition~\ref{obvious}~(\ref{t2}) and Theorem~\ref{iff}.
\end{proof}

\begin{cor}\label{main1}
    Let $M$ be a closed essential $n$-manifold with $\pi:=\pi_1(M)$. If $B\pi$ is a closed $n$-manifold, then $\dcat(M) = \cat(M)=n$.
\end{cor}
\begin{proof}
Since $M$ is essential, we have $f_*([M])\ne 0\in H_n(B\pi;\Z_w)$ for a classifying map $f:M\to B\pi$ as above. This follows from~\cite[Theorem 8.2]{Ba}. For the case when $M$ is orientable, this is proved in~\cite[Proposition 3.2]{BD} as well. We note that the proof of~\cite[Proposition 3.2]{BD} also works if $M$ is non-orientable since $f$ takes the orientation sheaf $\O_\pi$ on the closed $n$-manifold $B\pi$ to the orientation sheaf $\O_M$ on $M$. This happens because $\O_M$ and $\O_\pi$ are obtained as pullbacks of the canonical $\Z$-bundle $\O$ on $\R P^{\infty}$ along the maps $w_M:M\to \R P^{\infty}$ and $w_\pi:B\pi\to\R P^\infty$, respectively, that represent the respective first Stiefel--Whitney classes (see~\cite[Section 2.4]{Dr1}), and $w_\pi\circ f=w_M$. In particular, $\text{deg}(f)\ne 0$. Then, since the primary obstruction $\kappa_n\in H^n(B\pi;\F)$ to a section of the fibration $\B_n(p_\pi):P(B\pi)_n\to B\pi$ is also non-zero, we get that
\[
f_*([M])\frown\kappa_n  = \text{deg}(f)[B\pi]\frown\kappa_n = \text{deg}(f)\left([B\pi]\frown\kappa_n\right)\ne 0
\]
due to Poincaré duality in $B\pi$ in $\O_\pi\otimes\F$ coefficients. Thus, $M$ satisfies the cap property. Hence, the conclusion follows from Theorem~\ref{1}.
\end{proof}

This recovers the following statement, which is a consequence of Theorem~\ref{eg}.

\begin{cor}\label{kwobv}
    If $M$ is a closed aspherical $n$-manifold, then $\dcat(M) =n$.
\end{cor}

\subsection{Symplectically aspherical manifolds} To emphasize the utility of Corollary~\ref{main1}, we will present in this section various examples of non-aspherical essential manifolds whose distributional category turns out to be maximum.

Let $(M,\omega)$ be a closed symplectic manifold of dimension $2n$ such that $[\omega]^n \ne 0$, where $[\omega]\in H^2(M;\R)$ is the de Rham cohomology class corresponding to the non-degenerate symplectic $2$-form $\omega$. If $(M,\omega)$ is simply connected, then clearly, $n \le c\ell_{\R}(M) \le \dcat(M) \le \cat(M) \le n$ implies $\dcat(M) = n$. So, we assume that $M$ is not simply connected. Let $c_1 \in H^2(M;\Z)$ denote the first Chern class of $M$ and let $h:\pi_2(M) \to H_2(M;\Z)$ denote the Hurewicz homomorphism. We follow~\cite{Ori} and call $(M,\omega)$ \emph{spherically monotone} if there exists some $\lambda \in \R$ such that the composition 
\[
([\omega]-\lambda c_1)\circ h:\pi_2(M) \to H_2(M;\Z) \to \R
\]
is trivial. When $\lambda = 0$, then $(M,\omega)$ is called \emph{symplectically aspherical}.
Examples of these kinds include symplectic manifolds whose second homotopy groups vanish, such as aspherical manifolds. 

\begin{cor}\label{sympl}
    Let $(M,\omega)$ be a spherically monotone $2n$-manifold with some $\lambda\in\R$ such that $([\omega]-\lambda c_1)^n\ne 0$. If $B\pi_1(M)$ is a closed $2n$-manifold, then $\dcat(M) = 2n$.
\end{cor}

\begin{proof}
    When $\lambda=0$, then $M$ is essential due to~\cite[Corollary 4.2]{RO}. In general, $M$ is essential by~\cite[Theorem 4.6]{Ori} if $([\omega]-\lambda c_1)^n\ne 0$. So, Corollary~\ref{main1} applies.
\end{proof}

\begin{remark}
    We note that the LS-category of these symplectic manifolds was calculated in~\cite{RO} and~\cite{Ori} using the notion of category weight (see~\cite{St},~\cite{R1}). For $\dcat$, the author has not been able to develop a ``useful" notion of ``distributional category weight" that can determine $\dcat(M)$. However, our current technique does give many new computations for $\dcat$ of symplectically aspherical manifolds.
\end{remark}

We now focus on closed symplectically aspherical (SA) $2n$-manifolds $(M,\omega)$ for which $B\pi_1(M)$ are closed $2n$-manifolds. 

\begin{ex}[\protect{Aspherical}]\label{asph}
    Let $\mathscr{A}$ denote the class of closed aspherical manifolds that are symplectic. Interesting examples in this class include the following. For each $n \ge 1$, let $\P_n = \prod_{i=1}^n \Sigma_{g_i}$, where $\Sigma_{g_i}$ is the closed orientable surface of genus $g_i \ge 1$. Then $\P_n$ is an aspherical symplectic $2n$-manifold. Thus, $\dcat(\P_n) = 2n$. 
     
    Let $H_3(\R)$ and $H_3(\Z)$ denote, respectively, the real and the discrete Heisenberg groups. Then $\H_3 = H_3(\R)/H_3(\Z)$ is the Heisenberg $3$-manifold whose fundamental group is $H_3(\Z)$. Let $\mathcal{KT} = \H_3 \times S^1$ denote the Kodaira--Thurston manifold. It is an aspherical symplectic $4$-manifold~\cite{RT}. Thus, $\dcat(\mathcal{KT}) = 4$. 
    
    Let $G$ be a simply connected completely solvable Lie group of dimension $2n$ and $\Gamma$ be a co-compact lattice in $G$. Then the quotient $G/\Gamma$ is a closed aspherical $2n$-manifold. If there exists $\alpha\in H^2(G/\Gamma;\R)$ such that $\alpha^n\ne 0\in H^{2n}(G/\Gamma;\R)$, then $G/\Gamma$ has a symplectic structure by~\cite[Lemma 2.2]{IKRT} and thus, it is SA. Hence, $\dcat(G/\Gamma) = 2n$ is obtained. 
    \end{ex}

However, to get new computations of $\dcat$ using Corollary~\ref{sympl}, we are interested in non-aspherical SA manifolds (while Gompf~\cite{Go} provides such examples in dimension $4$, we don't know if they satisfy the cap property and so, their $\dcat$ cannot be determined by our methods). Indeed, many such manifolds exist in all even dimensions $\ge 4$ and we obtain new computations of $\dcat$ in all those cases.

\begin{ex}[\protect{In the fourth dimension}]\label{exist}
    Due to~\cite[Corollary 5.4]{KRT}, there exists an SA $4$-manifold $\M^4$ such that $\pi_1(\M^4) = \Z^4$ and $\pi_2(\M^4) \ne 0$. Thus, we get $\dcat(\M^4) = 4$ from Corollary~\ref{sympl}. Let the class of such $4$-manifolds be denoted by $\mathscr{B}$. In Example~\ref{prod}, we shall give such examples in all even dimensions.
\end{ex}

\begin{ex}[\protect{Branched covers}]\label{cover}
    For each fixed $n \ge 2$, let $Y^n = \C^n/\Lambda$ be an abelian variety of complex dimension $n$ for $\Lambda = \Z^{2n}$ ($Y^n$ is a projective $2n$-torus). In particular, $Y^n$ is SA and $\pi_1(Y^n) = \Z^{2n}$. In~\cite[Section 3]{Wa} and~\cite[Section 2]{DCP}, using an ample line bundle $L$ of $Y^n$ and choosing a smooth divisor $B^{n-1}$ in the linear system $|2L|$, a degree two branched cover $f: X^n \to Y^n$ is defined with $B^{n-1}$ as the branch locus. These branched covers are interesting objects in algebraic geometry, see~\cite{Wa},~\cite{DCP}. In particular, $X^n$ is not aspherical due to~\cite[Theorem 2.2]{DCP}. Here, $B^{n-1}$ can be chosen to be a (real) codimension $2$ symplectic submanifold of $Y^n$. So, $X^n$ is an SA $2n$-manifold by~\cite[Lemma 1 and Theorem 3]{Go}. Finally, we note that $\pi_1(X^n) = \pi_1(Y^n) = \Z^{2n}$ due to~\cite[Corollary 3.4]{PT} (see also the proof of~\cite[Theorem 2.2]{DCP}). Hence, Corollary~\ref{sympl} gives $\dcat(X^n) = 2n$. Let the class of such non-aspherical SA branched covers be denoted by $\mathscr{C}$. 
\end{ex}

\begin{ex}[\protect{Products}]\label{prod}
For each $n\ge 1$, we compute $\dcat$ of SA $(2n+4)$-manifolds $M$ for which $\pi_2(M)\ne 0$. Let $\P_n,G/\Gamma\in\mathscr A$ and $\mathcal{M}^4 \in \mathscr{B}$.
    \vspace{-2.5mm}
\begin{enumerate}
    \itemsep-0.3em 
    \item If $M = \P_n \times \M^4$, then $B\pi_1(M) = \P_n \times T^4$. Hence, $\dcat(M)=2n+4$.
        \item If $M = G/\Gamma \times \M^4$, then $B\pi_1(M) = G/\Gamma \times T^4$. Hence, $\dcat(M) =2n+4$.
\end{enumerate}
    \vspace{-1mm}
Here, $T^4$ denotes the $4$-torus. In general, finite products of manifolds in classes $\mathscr{B}$ and $\mathscr{C}$, and their products with manifolds in class $\mathscr{A}$, will give non-aspherical SA manifolds whose $\dcat$ values will be maximum as a result of Corollary~\ref{sympl}.
\end{ex}

\begin{ex}[\protect{From symplectic Lefschetz fibrations}]\label{fibersum}
    For $g \ge 1$, let $\Gamma_g: = \pi_1(\Sigma_g)$. Note that $\Gamma_g$ is finitely presented and the identity map $i_g:\Gamma_g\to\Gamma_g$ is a finite presentation. Clearly, the induced map $i_g^*:H^2(\Gamma_g;\R) \to H^2(\Gamma_g;\R)$ is non-zero. So, by~\cite[Proposition 3.1]{KRT}, there exists a symplectic Lefschetz fibration
    \[
    \Sigma_h \hookrightarrow X  \xrightarrow{f} S^2
    \]
    for some $h \ge g$ with $\pi_1(X) = \Gamma_g$ and $\Omega\in H^2(X;\R)$ such that $\langle \Omega,\phi_*[S^2] \rangle = 0$ and $\psi_y^*(\Omega) \ne 0$ for all smooth maps $\phi:S^2 \to X$ and $\psi_y:f^{-1}(y) \hookrightarrow X$ for $y \in S^2$. For any $k \ge 1$, let $Y_k = \Sigma_h \times \Sigma_k$. Then the Gompf symplectic fiber sum $X \#_{\Sigma_h} Y_k$ (see~\cite[Section 3]{IKRT} for its definition and also Section~\ref{6.1} for a general description of fiber sums) is SA by~\cite[Proposition 3.3]{KRT}. Due to~\cite[Corollary 4.5]{KRT}, 
    \[
    \pi_1\left(X \#_{\Sigma_h} Y_k\right) = \Gamma_g \oplus \Gamma_k.
    \]
    Hence, we apply Corollary~\ref{sympl} to get $\dcat(X \#_{\Sigma_h} Y_k) = 4$. We note that in general, the fiber sums $X \#_{\Sigma_h} Y_k$ need not be aspherical.
\end{ex}

We now explain how for any $n \ge 2$, any direct sum of $n$-number of surface groups (i.e., fundamental groups of closed orientable surfaces $\Sigma_g$ of genus $g \ge 1$) can be realized as the fundamental group of a closed symplectically aspherical $2n$-manifold $M$ that is not aspherical but whose $\dcat(M) = \cat(M)=2n$.

\begin{remark}
    For any $n \ge 2$, let $\Gamma^n := \oplus_{i=1}^n \Gamma_{g_i}$, where $\Gamma_{g_i} = \pi_1(\Sigma_{g_i})$ for $g_i \ge 1$. We saw in Examples~\ref{exist} and~\ref{prod} that if at least two summands in $\Gamma^n$ are $\Z^2$, then there exists a non-aspherical SA $2n$-manifold $M$ such that $\pi_1(M) = \Gamma^n$ and $\dcat(M) = 2n$. Now, let us assume that at most one summand in $\Gamma^n$ is $\Z^2$. Then there exists $j \le n$ such that $g_j \ge 2$. In Example~\ref{fibersum}, we take $g = g_j$ and $k = g_s$ for any fixed $s \ne j$ and get $M_{js} = X \#_{\Sigma_h}Y_{g_s}$ for some $h \ge g_j$. If $n \ge 3$, then we let $\P_{n-2} = \prod_{i=1, i \ne j,s}^n\Sigma_{g_i}$ and consider $M = M_{js} \times \P_{n-2}$. Whenever $M_{js} \ne \Sigma_{g_j} \times \Sigma_{g_s}$, it is not aspherical. So, in those cases, $M$ is a non-aspherical SA $2n$-manifold such that $\pi_1(M) = \Gamma^n$ and $\dcat(M) = 2n$ due to Corollary~\ref{sympl}.
    \end{remark}

\section{For connected sums}\label{connsums}
Let $M$ and $N$ be closed $n$-manifolds and $M\# N$ be their connected sum. An analog of the formula 
\[
\cat(M\#N) = \max\{\cat(M),\cat(N)\}
\]
from~\cite{DS} for the LS-category of connected sum does not hold for $\dcat$, at least when both $\pi_1(M)$ and $\pi_1(N)$ have torsion --- see~\cite[Remark 6.5]{DJ} for a non-orientable example when $M=N=\R P^2$ and Remark~\ref{notinorient} for an orientable example when $M=N=\R P^3$. It is unclear if such a general connected sum formula holds for $\dcat$ if $\pi_1(N)$ is torsion-free. In this section, we provide such a partial formula when $N$ is essential and satisfies the cap property defined in Section~\ref{torsionfree}.

Let $\O_{M}$ and $\O_N$ be the orientation sheaves on the closed $n$-manifolds $M$ and $N$, respectively. Let $B$ denote a closed $n$-disc common to $M$ and $N$, whose interior is deleted from $M$ and $N$ to identify the resulting boundary spheres and form the connected sum $M\# N$. If $\O_B$ is the orientation sheaf on $B$, then we can identify $\O_M$ and $\O_N$ along $\O_B$ to get the orientation sheaf $\O'$ on $M \cup_B N$. Here, $M \cup_B N$ denotes the union of $M$ and $N$ along the closed $n$-disc $B$. Consider the quotient map $\psi:M\cup_BN\to(M\cup_BN)/B = M \vee N$. Since $B$ is contractible, $\psi$ admits a homotopy inverse $\iota:M \vee N \to M\cup_BN$. Then the orientation sheaf $\O$ on $M\vee N$ is obtained by pulling back $\O'$ along $\iota$. If $q:M\# N\to M\cup_B N$ denotes the inclusion, then the restriction of $\O'$ along $q$ gives the orientation sheaf $\O_{\#}$ on $M\# N$. So, in particular, the map $\psi\circ q$ takes $\O$ to $\O_{\#}$.

\begin{proof}[Proof of Theorem~\ref{2}]
     We follow the above notations. Let $d: M \vee N \to N$ and $\phi:M\#N \to N$ be the maps collapsing $M$ to a point, and let $g:N \to B\pi_1(N)$ be a classifying map that cannot be deformed into the $(n-1)$-skeleton of $B\pi_1(N)$. Let us consider the following commutative diagram.
\[
\begin{tikzcd}[contains/.style = {draw=none,"\in" description,sloped}]
&
M\# N \arrow{r}{q} \arrow[swap]{d}{\phi} 
& 
M\cup_B N \arrow{d}{\psi}
\\
B\pi_1(N)
&
N   \arrow{l}[swap]{g}
& 
M \vee N \arrow{l}[swap]{d}
\end{tikzcd}
\]
By definition, $\phi_*([M\#N])=[N]$ and $q_*([M\#N]) = [M]\oplus [N]$. Thus, $q$ induces an epimorphism of the fundamental groups. So, in view of the homotopy equivalence $\psi$, the induced map 
\[
    (\psi\circ q)_*:H_0(M\# N;(\psi\circ q)^*(\A)) \to H_0(M\vee N;\A)
    \]
is an isomorphism for each coefficient system $\A$. Note that $(\psi\circ q)^*(\O)=\O_{\#}$. Let $\kappa_n\in H^n(B\pi_1(N);\F)$ be the primary obstruction to a section of the fibration  
$\B_n(p):P(B\pi_1(N))\to B\pi_1(N)$, see Paragraph 1 of Section~\ref{torsionfree} for details. Then 
\[
(g\circ\phi)^*(\kappa_n) \in H^n(M\#N;(g\circ\phi)^*(\F))
\]
is the primary obstruction to the lift of $g\circ\phi$ with respect to the fibration $\B_n(p)$. Since $\dim(M\# N)=n$, it is the only obstruction to such a lift of $g\circ \phi$. Clearly, we have $d^*(g^*(\kappa_n)) = \alpha \oplus g^*(\kappa_n)$ for some cohomology class $\alpha$. Since $\phi = d \circ \psi \circ q$ and $(\psi \circ q)_*([M\#N]) = [M]\oplus [N]$, we get the following set of equalities in $\O\otimes \hspace{0.5mm}(g\circ d)^*(\F)$ coefficients:
    \begin{multline*}
(\psi \circ q)_*\left([M\#N]\frown \phi^*(g^*(\kappa_n))\right)
= (\psi \circ q)_*\left([M\#N]\frown (\psi \circ q)^* \left(d^*(g^*(\kappa_n))\right)\right)
\\
= (\psi \circ q)_*([M\# N])\frown d^*(g^*(\kappa_n))
= ([M]\oplus [N])\frown(\alpha \oplus g^*(\kappa_n)).
\end{multline*}
Let $w:\pi_1(N)\to\Z_2=\text{Aut}(\Z)$ be the first Stiefel--Whitney class and $\Z_w$ be the twisted integers so that $[N]\in H_n(N;\O_N)= H_n(N;\Z_w)$, see Remark~\ref{formal}. Then, by the cap property of $N$, we have in $\Z_w\otimes\F$ coefficients that
\[
g_*([N]\frown g^*(\kappa_n)) = g_*([N]) \frown \kappa_n \ne 0.
\]
In particular, $[N]\frown g^*(\kappa_n) \ne 0$. Therefore, $([M]\oplus [N])\frown(\alpha \oplus g^*(\kappa_n)) \ne 0$ and thus, the first equality above implies that $\phi^*(g^*(\kappa_n)) = (g\circ\phi)^*(\kappa_n) \ne 0$. Hence, $g\circ \phi$ does not admit a lift with respect to $\B_n(p)$. By Propositions~\ref{obvious}~(\ref{two}) and~\ref{lift}, we get that 
\[
n \le \dcat(g\circ\phi)\le \dcat(M\#N) \le \cat(M\#N) = n.
\]
Thus, $\dcat(M\#N) = n$. 
\end{proof}

It follows from Theorem~\ref{2} and the proof of Corollary~\ref{main1} that if $N$ is a closed essential $n$-manifold such that $B\pi_1(N)$ is a closed $n$-manifold, then $\dcat(M\#N) = n$ for each closed $n$-manifold $M$. Therefore, we can directly use our symplectically aspherical manifolds from Examples~\ref{exist},~\ref{cover}, and~\ref{prod} to get more examples of manifolds whose $\dcat$ will be maximum.

\begin{ex}
    If $M^k$ is a closed $k$-manifold, $\P_n\in\mathscr A$, $\M^4\in\mathscr B$, and $X^n\in\mathscr{C}$, then
        \vspace{-6mm}
\begin{enumerate}
        \itemsep-0.35em 
        \item $\dcat(M^4 \# \M^4) = 4$.
        \item For each $n\ge 2$, $\dcat(M^{2n}\# X^n) = 2n$.
        \item For each $n\ge 1$, $\dcat(M^{2n+4}\# (\M^4\times\P_n)) = 2n+4$.
    \end{enumerate}
\end{ex}

\begin{cor}\label{realmain2}
    Let $M$ and $N$ be closed $n$-manifolds. If $N$ is aspherical, then we have $\dcat(M \# N) = \cat(M \# N) = n$.  
\end{cor}

\begin{proof}
    Since $N$ is aspherical, we can take $g$ as the identity map in the proof of Theorem~\ref{2}. Using Poincaré duality in $N$ in $\O_N \otimes \F$ coefficients, we then get $n \le \dcat(\phi) \le \dcat(M\# N) \le n$ by proceeding in the same way as above.
\end{proof}

We note that just like Corollary~\ref{main1}, Corollary~\ref{realmain2} also recovers Corollary~\ref{kwobv}: if $N$ is a closed aspherical $n$-manifold, then $\dcat(N) = \dcat(S^n\#N) = n$. In fact, Theorem~\ref{1} is obtained as a special case of Theorem~\ref{2}.

\begin{cor}\label{result2}
If $M$ is a closed rationally essential $3$-manifold, then $\dcat(M)=3$.
\end{cor}

\begin{proof}
If $M$ is a closed rationally essential $3$-manifold, then we can write $M=P\# N$ for a closed $3$-manifold $P$ and a closed aspherical $3$-manifold $N$, see~\cite[Theorem 3]{KN}. Then, Corollary~\ref{realmain2} implies $\dcat(M)=3$.
\end{proof}

We note that Corollary~\ref{result2} is not true for general closed essential $3$-manifolds in view of the fact that $\dcat(\R P^3)=1$.

\begin{remark}
    For each $n\ge 1$, let $T^n$ denote the $n$-torus. Due to Corollary~\ref{realmain2}, we have $\dcat(T^n\# T^n) = n$.
    However, $B\pi_1(T^n\# T^n)=T^n\vee T^n$ is not a closed $n$-manifold for any $n$. Thus, the converse of Corollary~\ref{main1} is not true.
\end{remark}

For non-aspherical spaces, their rational (or real) cup-lengths were typically used to determine their $\dcat$ values in~\cite[Section 6.1]{DJ}. Given a connected sum, the cup-lengths of its summands can be used to estimate its $\dcat$. Indeed, for $\mathbb{F}\in\{\Q,\R\}$, we have that
\[
\max\{c\ell_{\mathbb{F}}(M),c\ell_{\mathbb{F}}(N)\} \le \dcat(M\#N) \le \cat(M\# N),
\]
see~\cite[Remark 8.2]{J2}. 
For example, $n\le \dcat(T^n\#T^n) \le \cat(T^n\#T^n)=n$ implies $\dcat(T^n\#T^n)=n$.
But from this technique alone, we cannot prove that $\dcat(M\# N)$ is maximum unless
$c\ell_{\mathbb{F}}(M)=\cat(M\# N)$ or $c\ell_{\mathbb{F}}(N)=\cat(M\# N)$. So, for an aspherical $n$-manifold $N$, Corollary~\ref{realmain2} gives new computations $\dcat(N\# N)$ in the case when $c\ell_{\mathbb{F}}(N) < \cat(N)$ for $\mathbb{F}\in\{\Q,\R\}$. We now discuss several such examples.

\begin{ex}[\protect{From non-orientable surfaces}]
    Let $N_{g_i}$ be a closed non-orientable surface of genus $g_i \ge 2$. Then for any $m \ge 1$, the product $Q_m=\prod_{i=1}^mN_{g_i}$ is a closed apsherical $2m$-manifold such that $c\ell_{\mathbb{F}}(Q_m) < 2m$. By Corollary~\ref{realmain2}, $\dcat(M^{2m}\#Q_m) = 2m$ for any closed $2m$-manifold $M^{2m}$. 
    More generally, for any $n,k \ge 1$, Corollary~\ref{realmain2} gives $\dcat(M^{2m+2n}\#(Q_m\times\P_n)) = 2m+2n$ and $\dcat(M^{2m+k}\#(Q_m\times T^k)) = 2m+k$, where $T^k$ is the $k$-torus and $\P_n$ is as in Example~\ref{asph}, even though $c\ell_{\mathbb{F}}(Q_m \times \P_n) < 2m+2n$ and $c\ell_{\mathbb{F}}(Q_m \times T^k) < 2m+k$.
\end{ex}

\begin{ex}[\protect{Subgroups of Lie groups}]
    Let $G$ be a Lie group and $H$ be its maximal compact subgroup. If $\Gamma$ is a torsion-free lattice in $H$, then $N = \Gamma\backslash G/H$ is a closed aspherical manifold of dimension $n$ for some $n \ge 1$ depending on the choice of $H$ and $\Gamma$. So, $\dcat(M^n \# N) = n$ by Corollary~\ref{realmain2}.
\end{ex}

\begin{ex}[\protect{Nilmanifolds}]
Let $G$ be a simply connected nilpotent Lie group of dimension $n$ and $\Gamma$ be a non-abelian co-compact lattice in $G$. Then $N=G/\Gamma$ is a closed nilmanifold and $c\ell_{\mathbb{F}}(N) <  n = \dim(N)$, see~\cite[Proposition 4.3]{RT}.
Hence, $\dcat(M^n\# N) = n$ by Corollary~\ref{realmain2}. In particular, we get $\dcat(N\# N) = n$ as new computations. An example of this kind is the Heisenberg manifold $\H_3$ described in Example~\ref{asph}. We note that for the Kodaira--Thurston manifold $\mathcal{KT} = \H_3 \times S^1$, we have $c\ell_{\mathbb{F}}(\mathcal{KT}) = 3 < 4$~\cite{RT}. Thus, $\dcat(M^4\#\mathcal{KT})=4$, and $\dcat(\mathcal{KT}\# \mathcal{KT}) = 4$ is another new computation due to Corollary~\ref{realmain2}.
\end{ex}

\begin{ex}[\protect{Hantzsche--Wendt manifolds}]\label{hwmani}
For odd $n \ge 3$, a flat orientable $n$-manifold $\mathcal{N}$ is called a \emph{generalized Hantzsche--Wendt $n$-manifold} if its holonomy group is $\Z_2^{n-1}$. When $n=3$, this is the well-known Hantzsche--Wendt manifold (see, for example,~\cite{hw}), which is the union along the boundaries of twisted orientable interval-bundles over a Klein bottle (described explicitly in Section~\ref{vsimple} as $\mathcal{M}_6$). In each odd dimension $\ge 3$, $\mathcal{N}$ is a rational homology sphere due to~\cite{hw}. Therefore, $c\ell_{\Q}(\mathcal{N})=1 < n = \cat(\mathcal{N})$. By Corollary~\ref{realmain2}, $\dcat(M^n\#\mathcal{N}) = n$, and so, $\dcat(\mathcal{N}\#\mathcal{N}) = n$ are new computations. In fact, for any $m \ge 1$, if $T^m$ is the $m$-torus, then $c\ell_{\Q}(\mathcal{N} \times T^m) < n+m$. Hence, we get
\[
\dcat\left(\left(\mathcal{N} \times T^m\right)\#\left(\mathcal{N} \times T^m\right)\right) = n+m
\]
as a new computation in each dimension $\ge 4$ by Corollary~\ref{realmain2}.
\end{ex}

\section{For low-dimensional manifolds}\label{lowdim}
In this section, we focus on closed manifolds of dimensions $n\le 4$. In dimensions $n \le 2$, the $\dcat$ values are completely known,~\cite[Section 6.1]{DJ} (see also Section~\ref{Introduction}). So, the first interesting problem in determining $\dcat$ of low-dimensional manifolds comes in dimension $3$. 

\subsection{$3$-manifolds}\label{threemani}
The classical LS-category of closed $3$-manifolds was computed first in~\cite{GLGA} and then in~\cite{OR} using category weight. The complete description based on the fundamental groups is as follows. Let $M$ be a closed $3$-manifold. Then,
\[
\cat(M) =\begin{cases}
    1 & \textup{if } \pi_1(M) = 0 \\
    2 & \textup{if } \pi_1(M)\ne 0 \text{ is free} \\
    3 & \textup{if } \pi_1(M) \text{ is not free}
\end{cases}
\]

Using some ideas from~\cite{OR} and Theorem~\ref{1}, we can compute the distributional category of any closed $3$-manifold $M$ such that $\pi_1(M)$ is torsion-free. 

\begin{proof}[Proof of Theorem~\ref{bb}]
    First, let $M$ be simply connected. In this case, we have that $\dcat(M) \le \cat(M) = 1 \implies \dcat(M) = 1$.\\

    Now, let us assume that $\pi_1(M) \ne 0$ is free. Let $M = M_1\#\cdots\#M_k$ be the unique (up to permutation) decomposition of $M$ into prime $3$-manifolds $M_i$, see~\cite[Theorems 3.15 and 3.21]{He}. Since $\pi_1(M) \ne 0$ is free, $\pi_1(M_i)$ is free for each $1 \le i \le k$ and there exists $j$ such that $\pi_1(M_j) \ne 0$. Due to~\cite[Corollary 4.6]{OR}, $M_j$ cannot be irreducible. Therefore, $M_j$ must be an $S^2$-bundle over $S^1$ by~\cite[Lemma 3.13]{He}. Let us write $M = N \# M_j$, where $N$ is the connected sum of the other prime components (if $k=1$, we take $N = S^3$). If $M_j$ is orientable, then $M_j = S^2 \times S^1$. Therefore, $2 = c\ell_{\Q}(S^2\times S^1) \le \dcat(N\# M_j)$. If $M_j$ is non-orientable, then $S^2 \times S^1$ is a double cover of $M_j$. Hence, $N\#N\#(S^2\times S^1)$ is a cover of $N \# M_j$. Using the rational cup-length of $S^2 \times S^1$ in the same way as before and Theorem~\ref{covv}, we get
    \[
    2 = c\ell_{\Q}(S^2\times S^1) \le \dcat(N\# N\#(S^2\times S^1)) \le \dcat(N\# M_j).
    \]
    In any case, we have $2 \le \dcat(M) \le \cat(M) = 2$. Thus, $\dcat(M) = 2$. \\

    Finally, let $\pi_1(M)$ be not free but torsion-free. Let $M = M_1\#\cdots\#M_k$ be the (essentially unique) decomposition of $M$ into prime $3$-manifolds $M_i$. Since $\cat(M) = 3$, due to the connected sum formula~\cite{DS}, there exists $j$ such that $\cat(M_j) = 3$. Since the LS-category of $S^2$-bundles over $S^1$ is $2$, $M_j$ must be irreducible by~\cite[Lemma 3.13]{He}. Since $\pi_1(M)$ is torsion-free, so is $\pi_1(M_j)$. In particular, $\pi_1(M_j)$ is infinite. Thus, $M_j$ is not a finite quotient of $S^3$. This and the irreducibility of $M_j$ implies that $M_j$ is aspherical. 
    Hence, we have $M = N\#M_j$, where $N$ is the sum of the other summands and $M_j$ is aspherical. So, we apply Corollary~\ref{realmain2} to get $\dcat(M) =3$.
\end{proof}

\begin{remark}\label{rp3}
In Theorem~\ref{bb}, we cannot drop the torsion-free hypothesis on $\pi_1(M)$ when $\pi_1(M)$ is not free: $\pi_1(\R P^3) = \Z_2$ is not free but $\dcat(\R P^3) = 1$.
\end{remark}

The only remaining case is when $\pi_1(M)$ has torsion (this implies that $\pi_1(M)$ is not free). We say that $M \in \mathscr{D}$ if $\pi_1(M)$ has torsion. To explore the possibility of determining $\dcat$ of $3$-manifolds in the class $\mathscr{D}$, we consider the following examples.

\begin{ex}\label{rp3sum}
    Clearly, $\R P^3 \#\R P^3\in\mathscr{D}$. We can view $\R P^3 \#\R P^3$ as the quotient space $(S^2 \times [0,1])/\sim$, where $(x,0)\sim (-x,0)$ and $(x,1)\sim (-x,1)$ for all $x\in S^2$. Then, $S^2 \times S^1$ is a double cover of $\R P^3 \#\R P^3$. By Theorem~\ref{covv}, we get that $2 = \dcat(S^2\times S^1) \le \dcat(\R P^3 \#\R P^3)$. 
\end{ex}

\begin{remark}\label{notinorient}
    Due to Example~\ref{rp3sum}, $\dcat(\R P^3) = 1 < 2 \le \dcat(\R P^3 \#\R P^3)$. Hence, the connected sum formula from~\cite[Theorem 1]{DS} does not hold for $\dcat$ even when all the summands are orientable manifolds.
\end{remark}

\begin{ex}\label{lens}
If $M$ and $N$ are two non-trivial $3$-dimensional lens spaces, then $M\#N\in\mathscr{D}$. Depending on the lens spaces chosen, there exists some $n\ge 2$ such that $\#_{i=1}^n(S^2\times S^1)$ is a cover of $M\# N$. Therefore, $\dcat(M\#N)\ge 2$. 
\end{ex}

\begin{ex}\label{sumwithrp3}
    Let $M$ be a $3$-manifold admitting the spherical geometry and let $N$ be aspherical. Since $\pi_1(M)$ is finite, $M\#N\in \mathscr{D}$. However, Corollary~\ref{realmain2} gives $\dcat(M\# N) = 3$. So, the converse of Theorem~\ref{bb} is not true.
\end{ex}

Because of $\R P^3$ and the above examples, we see that manifolds in class $\mathscr{D}$ cannot have a constant $\dcat$ value. If $M \in \mathscr{D}$, then while $\cat(M) = 3$, one cannot estimate $\dcat(M)$ in general. So, Theorem~\ref{bb} is optimal in some sense.

Due to Thurston, $\mathbb{E}^3$, $\N il^3$, $\mathbb{H}^2\times\mathbb{E}$, $\wt{\mathbb{SL}}$, $\mathbb{H}^3$, $\mathbb{S}ol^3$, $\mathbb{S}^3$, and $\mathbb{S}^2\times\mathbb{E}^1$ are the eight maximal $3$-dimensional geometries. If $M$ has one of the first six geometries, then $M$ is aspherical, and thus, $\dcat(M) = 3$. By Corollary~\ref{realmain2}, if $M$ appears in the prime decomposition of a closed $3$-manifold $K$, then $\dcat(K)=3$. If $M$ has the spherical geometry $\mathbb{S}^3$, then $M\in\mathscr D$ and we do not know $\dcat(M)$ in general. However, if $M$ has $\mathbb{S}^2\times\mathbb{E}^1$ geometry, then $\dcat(M) \ge 2$ because there are only the following four possibilities for $M$.
\begin{itemize}
            \itemsep-0.3em 
\item $M = S^2 \times S^1$. Then $\dcat(M) = 2$ by~\cite[Proposition 6.7]{DJ}.
    \item $M$ is the non-orientable $S^2$-bundle over $S^1$. Then $\dcat(M)=2$ by the proof of Theorem~\ref{bb}.
    \item $M = \R P^3\#\R P^3$. Then $\dcat(M) \ge 2$ by Example~\ref{rp3sum}.
    \item $M = \R P^2\times S^1$. Then since $S^2\times S^1$ is a cover of $M$, Theorem~\ref{covv} directly gives $2=c\ell_{\Q}(S^2\times S^1)\le \dcat(M)$.
\end{itemize}
 So, if such a manifold $M$ appears in the prime decomposition of a $3$-manifold $K$, then $\dcat(K)\ge 2$ by the cup-length technique used in the proof of Theorem~\ref{bb}. 

\begin{remark}\label{justref}
    Determining the exact value $\dcat(\R P^2\times S^1)$ will be very interesting. If $\dcat(\R P^2\times S^1) = 2$, then since $\cat(\R P^2\times S^1) = 3$, we will have another example of a closed manifold besides the real projective spaces whose $\dcat$ value will be less than the $\cat$ value. However, if $\dcat(\R P^2\times S^1) = 3$, then this will eliminate the possibility of a product inequality for $\dcat$ (see~\cite[Question 3.4]{DJ}) and disprove an analog of the famous Ganea conjecture (see~\cite[Conjecture 1.40]{CLOT}) for $\dcat$. 
\end{remark}

\subsection{$4$-manifolds with free fundamental group}\label{fourmani}
Let $M$ be a closed $4$-manifold such that $\pi_1(M)$ is free. First, let $\pi_1(M) = 0$. Then $\cat(M)\le 2$ by~\cite[Theorem 1.50]{CLOT}. If $M\simeq S^4$, then $\dcat(M) = \cat(M)=1$. If $M \not\simeq S^4$ is smooth, then $\dcat(M) = \cat(M)=2$ by arguments similar to that of~\cite[Example 8.5]{J2}. 

\begin{prop}
    If $M$ is a closed orientable $4$-manifold such that $\pi_1(M) \ne 0$ is free, then $\dcat(M) = \cat(M)=2$.
\end{prop}
\begin{proof}
As observed in~\cite[Page 2]{DKR} and~\cite[Page 5]{OS}, if $\pi_1(M) \ne 0$ is free, then $\cat(M) \le 2$. Since $\pi_1(M)$ is free of rank $k \ne 0$, $H_1(M;\Z)$ is a free abelian group of rank $k\ge 1$. Because $M$ is orientable, $H^1(M;\Q) \ne 0$. Due to Poincaré duality, $H^3(M;\Q) \ne 0$ and the cup product pairing
\[
H^1(M;\Q)\otimes H^3(M;\Q) \to H^4(M;\Q)
\]
is non-singular,~\cite[Proposition 3.38]{Ha}. So, there exists non-trivial  cohomology classes $u\in H^1(M;\Q)$ and $v\in H^3(M;\Q)$ such that $u\smile v \ne 0$. Therefore, $2\le c\ell_{\Q}(M) \le \dcat(M) \le \cat(M) \le 2$. Hence, we have $\dcat(M)=2$.
\end{proof}

We note that $\cat(M) = 2$ holds for any closed $4$-manifold $M \not\simeq S^4$ with $\pi_1(M)$ free, without any assumption on orientability. When $\pi_1(M)\ne 0$, the above proof works with $\Z_2$ coefficients since $c\ell_{\Z_2}(M) \le \cat(M)$. Otherwise, if $\pi_1(M) = 0$, then the conclusion follows from~\cite[Corollary 8.2]{KR}.

We will see more computations for 
$\dcat$ of closed $4$-manifolds in Section~\ref{newsubsect}.

\subsection{An analog of Rudyak's conjecture}
The following conjecture is well-known in the theory of the classical LS-category of manifolds,~\cite{R1},~\cite{CLOT},~\cite{R2}.
\begin{conjec}[\protect{Rudyak}]
    Let $M$ and $N$ be two closed orientable $n$-manifolds and let $f:M \to N$ be a map of degree $\pm 1$. Then $\cat(M) \ge \cat(N)$.
\end{conjec}

This conjecture is quite natural and is known to be true in several cases. In particular, it holds for low-dimensional manifolds, see~\cite{R1} and~\cite[Section 6]{OR}.

\begin{theorem}[\protect{\cite{R2}}]\label{cth}
     Let $M$ and $N$ be closed orientable $n$-manifolds for $n \le 4$ and let $f:M \to N$ be a map of degree $\pm 1$. Then $\cat(M) \ge \cat(N)$.
\end{theorem}

We aim to present an analog of this theorem for $\dcat$. First, we ask the following.

\begin{question}\label{conj}
    Let $M$ and $N$ be closed orientable $n$-manifolds and let $f:M \to N$ be a map of degree $\pm 1$. Is it true that $\dcat(M) \ge \dcat(N)$?
\end{question}

\begin{ex}
In the above setting, if $N = \R P^{n}$ for odd $n$, then Question~\ref{conj} is answered since $M$ is not contractible and thus, $\dcat(M) \ge 1 = \dcat(\R P^{n})$.    
\end{ex}

For $n\le 2$, $\dcat$ and $\cat$ of closed orientable $n$-manifolds coincide,~\cite[Section 6.1]{DJ}. So, Question~\ref{conj} is answered affirmatively for $n\le 2$ in view of Theorem~\ref{cth}.

\begin{remark}
The proof of Theorem~\ref{cth} for $n \in\{3,4\}$ uses the notions of category weight and detecting elements (see~\cite[Section 2]{R1}). Unfortunately, we do not have useful analogous notions in the distributional theory yet, so we attempt to answer Question~\ref{conj} affirmatively in some cases using our results from this paper.
\end{remark}

\begin{prop}
    Let $M$ and $N$ be two closed orientable $n$-manifolds and let $f:M\to N$ be a map of degree $\pm 1$. Then $\dcat(M)\ge \dcat(N)$ holds in each of the following four cases:
 \vspace{-2.5mm}
    \begin{enumerate}
            \itemsep-0.3em        
        \item $n = 3$ and $\pi_1(M)$ is torsion-free.\label{one1}
        \item $n=4$ and $M$ is smooth and simply connected.\label{two22}
        \item $n=4$ and $\pi_1(M)\ne 0$ is free.\label{two2}
        \item $n\in\{3,4\}$ and $M=K_1\#K_2$, where $K_2$ is essential, $\pi_1(K_2)$ is torsion-free, and $K_2$ satisfies the cap property (from Definition~\ref{capp}).\label{three3}
    \end{enumerate} 
\end{prop}

\begin{proof}
In case~(\ref{one1}), we get $\dcat(M) = \cat(M)$ due to Theorem~\ref{bb}. Similarly, in cases~(\ref{two22}) and~(\ref{two2}), $\dcat(M) = \cat(M)$ by our discussions in Section~\ref{fourmani}. In case~(\ref{three3}), we apply Theorem~\ref{2} in dimensions $3$ and $4$ to get $\dcat(M) = \cat(M)$ again. Thus, in all these cases, Theorem~\ref{cth} gives $\dcat(M) = \cat(M) \ge \cat(N) \ge \dcat(N)$.
\end{proof}

\section{For generalized connected sums}\label{gensum}

In this section, we will find sufficient conditions for $\dcat$ of some generalized connected sums (or fiber sums) of orientable manifolds to be maximum.

\subsection{Sums along submanifolds}\label{6.1}

Let $M$ and $N$ be two closed smooth orientable $n$-manifolds for $n \ge 3$. Let $S$ be a closed orientable manifold of dimension at most $n$. Suppose there are embeddings $i:S\to M$ and $j:S \to N$ with respective normal bundles $\mu$ and $\nu$ and an orientation-reversing bundle isomorphism $g:\mu\to\nu$. Such a map $g$ exists, in particular, when $\dim(S)=n-2$ and $e(\mu)=-e(\nu)$ for the Euler characteristics. Using smoothness, we can choose tubular neighborhoods $U\subset M$ and $V\subset N$ of $i(S)$ and $j(S)$, respectively, such that $U$ and $V$ are identified with the total spaces of $\mu$ and $\nu$, respectively. Then $g$ yields a diffeomorphism $h:U\to V$ that maps $i(S)$ to $j(S)$. This produces an orientation-reversing diffeomorphism $\phi:U - i(S) \to V -j(S)$, see~\cite[Section 3]{IKRT}. Let us denote this space (up to homeomorphism) by $L$. We now take the disjoint union 
\[
(M-i(S))\sqcup (N-j(S))
\]
and identify $U-i(S)$ with $V-j(S)$ using $\phi$ to obtain a closed orientable $n$-manifold $M\#_L\hspace{0.2mm}N$, which we call the \emph{generalized connected sum} of $M$ and $N$ along $L$.

Since $N$ is a closed manifold, $N/L$ is a compact pseudomanifold (with a CW structure) of dimension $n$ in the sense of~\cite[Section 2.4]{Dr1}. The orientation sheaf $\O_{N/L}$ on $N/L$ gives a fundamental class $[N/L]$ that defines the Poincaré duality homomorphism $\text{PD}_n$ from Section~\ref{basics} in dimension $n$. When $N/L$ is orientable, $\text{PD}_n$ is an isomorphism for each coefficient system, see~\cite[Theorem 2.4.1]{Dr1}.

\begin{proof}[Proof of Theorem~\ref{3}]
    First, we note that since $N$ and $N/L$ are closed and orientable of dimension $n$, $N/L$ has non-trivial cohomology in degree $n$. Since $N/L$ is aspherical as well, we have $\cd(\pi_1(N/L))\ge n$. Hence, $\dcat(N/L)\ge n$ by Theorem~\ref{eg}. Due to Proposition~\ref{characterization}, the fibration $\B_n(p):P(N/L)_n \to N/L$ does not admit a section. Hence, the primary obstruction $\kappa_n \in H^n(N/L;\F)$ to a section of $\B_n(p)$ is non-zero. Let us define the collapsing map 
    \[
    \phi:M\#_L\hspace{0.2mm}N\to (M\#_L\hspace{0.2mm}N)/M=N/L.
    \]
    By definition, $\phi^*(\kappa_n)\in H^n(M\#_L\hspace{0.2mm}N,\phi^*(\F))$ is the primary obstruction to a lift of $\phi$ with respect to $\B_n(p)$. Because $\dim(M\#_L\hspace{0.4mm}N)=n$, it is the only obstruction to such a lift of $\phi$. Due to orientability, $\phi$ sends $\O_{N/L}$ to the orientation sheaf $\O$ on $M\#_L\hspace{0.2mm}N$. Furthermore, $\text{deg}(\phi)=1$. Hence, in $\O_{N/L} \otimes \F$ coefficients, we get
\[
\phi_*\left([M\#_L\hspace{0.2mm}N] \frown \phi^*(\kappa_n)\right) = \phi_*([M\#_L\hspace{0.2mm}N]) \frown \kappa_n = [N/L] \frown \kappa_n \ne 0
\]
due to Poincaré duality in $N/L$ in the top dimension. This means $\phi^*(\kappa_n) \ne 0$ and so, a lift of $\phi$ does not exist with respect to $\B_n(p)$. Hence, using Propositions~\ref{obvious}~(\ref{t2}) and~\ref{lift}, we get the inequality
\[
n\le \dcat(\phi)\le \dcat(M\#_L\hspace{0.2mm}N) \le \cat(M\#_L\hspace{0.2mm}N)\le \dim(M\#_L\hspace{0.2mm}N)=n.
\]
\end{proof}

Such orientable aspherical pseudomanifolds can be obtained, for example, using the (relative) strict hyperbolization technique, see~\cite{DaJ},~\cite{CD},~\cite{Bel}. Let $X$ be a compact orientable $n$-manifold with boundary $\pa X\ne \emptyset$. Let $(R,\pa R)$ be the relative strict hyperbolization of $(X,\pa X)$. We note that $R$ is a compact orientable $n$-manifold,~\cite[1f]{DaJ}. Here, the $n$-pseudomanifold $R/\pa R$ can be seen as the strict hyperbolization of the space obtained from $X$ by attaching a cone over $\pa X$, see \cite[Page 4]{Bel}. Hence, $R/\pa R$ is closed and aspherical.

\begin{remark}
It is not known if the connected sum formula~\cite{DS} for classical LS-category ($\cat$) works for certain generalized connected sums as well, such as those studied in this section. Hence, Theorem~\ref{3} potentially provides new computations of $\cat$ for some special fiber sums.
\end{remark}

\subsection{Doubles of punctured manifolds}\label{gensum2}

We now consider a particular case of the generalized connected sum discussed above for compact manifolds. Let $M$ be a compact oriented $n$-manifold with boundary $\pa M\ne \emptyset$. We take a copy $\ov{M}$ of $M$ with reverse orientation and glue $M$ and $\ov{M}$ along $\pa M$ to form 
\[
\mathcal{D}M = M\#_{\pa M}\hspace{0.2mm}\ov{M}.
\]
This is a closed oriented $n$-manifold which we call the \emph{double of} $M$. 

Recall that any $k\ge 1$, a \emph{$k$-puncturing} of a closed orientable $n$-manifold $M$ refers to the process of the removal of $k$ number of small, disjoint open $n$-discs (or the interiors of closed $n$-discs) from $M$.

\begin{proof}[Proof of Theorem~\ref{4}]
Let $Z = \prod_{i=1}^n M_i$ and $m = \sum_{i=1}^n r_i$. For each $i\le n$, we can obtain $M_i$ back from $X_i$ by filling each of the $k_i$ punctures in $X_i$ by a suitable $r_i$-disc. Certainly, the inclusion $\mathcal{I}_i:X_i \hookrightarrow M_i$ is an embedding. So, there is a canonical map $\tau':M\to Z$. Let $q_1:M \to M/\pa M$ and $q_2:Z \to Z/\pa Z$ be the quotient maps that collapse the respective boundaries. Since $Z$ is closed, $q_2$ is the identity map on $Z$. Thus, $\tau'$ produces a natural map $\tau: M/\pa M \to Z$. Let 
\[
\phi:\mathcal{D}M \to \mathcal{D}M/M = M/\pa M
\]
be the collapsing map with $\text{deg}(\phi)=1$. Then, it follows that $(\tau \circ \phi)_*([\mathcal{D}M]) = [Z]$. Since $Z$ is aspherical of dimension $m$, the primary obstruction $\kappa_{m}\in H^{m}(Z;\F)$ to a section of the fibration $\B_{m}(p_Z):P(Z)_{m}\to Z$ is non-zero (see Section~\ref{torsionfree}). Clearly, 
\[
(\tau \circ \phi)^*(\kappa_{m})\in H^{m}(\mathcal{D}M;(\tau \circ \phi)^*(\F))
\]
is the primary obstruction to a lift of the map $\tau \circ \phi$ with respect to $\B_{m}(p_Z)$, and since $\dim(\mathcal{D}M)=m$, it is the only obstruction to such a lift of $\tau \circ \phi$. We note that $\tau \circ \phi$ induces an isomorphism 
\[
    (\tau \circ \phi)_*:H_0(\mathcal{D}M;(\tau \circ \phi)^*(\A)) \to H_0(Z;\A)
\]
for each coefficient system $\A$. Since $M$, $\mathcal{D}M$, and $Z$ are all orientable manifolds, $\tau \circ \phi$ sends the orientation sheaf $\O_Z$ on $Z$ to the orientation sheaf $\O$ on $\mathcal{D}M$. Thus, in $\O_Z\otimes\F$ coefficients, we have that
\[
(\tau \circ \phi)_*\left([\mathcal{D}M] \frown (\tau \circ \phi)^*(\kappa_{m})\right) = (\tau \circ \phi)_*([\mathcal{D}M]) \frown \kappa_{m} = [Z] \frown \kappa_{m} \ne 0
\]
due to Poincaré duality in $Z$. Thus, $(\tau \circ \phi)^*(\kappa_{m}) \ne 0$. Hence, we conclude that a lift of $\tau \circ \phi$ does not exist with respect to $\B_{m}(p_Z)$. Due to Propositions~\ref{obvious}~(\ref{t2}) and~\ref{lift}, we have that
\[
m\le \dcat(\tau \circ \phi)\le \dcat(\mathcal{D}M) \le m.
\]
Therefore, $\dcat(\mathcal{D}M) = \dim(\mathcal{D}M)=m$. 
\end{proof}

\begin{ex}
    If $X_i$ is a punctured closed nilmanifold of dimension $r_i$ for $1 \le i \le n$ and $M=\prod_{i=1}^n X_i$, then $\dcat(\mathcal{D}M) = \sum_{i=1}^n r_i\hspace{0.5mm}$ by Theorem~\ref{4}.
\end{ex}

\begin{remark}
    In the notations of Theorem~\ref{4}, more generally, we can take each $M_i$ to be a closed orientable essential $r_i$-manifold such that $B\pi_1(M_i)$ is a closed orientable $r_i$-manifold for all $1 \le i \le n$. Then we can proceed along the lines of Theorem~\ref{4} using a classifying map $g:Z\to \prod_{i=1}^nB\pi_1(M_i)$ and some ideas from Theorem~\ref{2} to obtain $\dcat(\mathcal{D} M) = \sum_{i=1}^n r_i$. So, in particular, we can take the non-aspherical symplectically aspherical branched covers $X^n$ from class $\mathscr{C}$ in Example~\ref{cover}, puncture them by deleting finitely many small, disjoint open $2n$-discs, take their product after puncturing, and conclude that the distributional category of the double of that punctured product is equal to its dimension.
\end{remark}

We mention that we have not been able to find results in the existing literature showing that the above \emph{double manifolds} are essential. So, we believe, Theorem~\ref{4} gives new computation of the classical LS-category of these manifolds using purely obstruction-theoretic techniques.

\subsection{Geometrically decomposable $4$-manifolds}\label{newsubsect}
We follow Hillman~\cite{Hi} and say that a $4$-manifold $M$ is \emph{geometrically decomposable} if it has a finite family $\mathcal{S}$ of disjoint $2$-sided hypersurfaces $S$ such that each component (called a \emph{piece}) of $M\setminus\cup S$ is geometric of finite volume. 
\begin{theorem}[\protect{\cite[Theorem 7.2]{Hi}}]\label{hillman}
    If a closed $4$-manifold $M$ admits a geometric decomposition, then exactly one of the following is true.
    \vspace{-2mm}
    \begin{enumerate}
                \itemsep-0.3em 
        \item $M$ is geometric.\label{tt1}
        \item $M$ is the total space of an $S^2$-orbifold bundle over a hyperbolic $2$-orbifold.\label{tt2}
        \item All components of $M\setminus \cup S$ have geometry $\mathbb{H}^2\times\mathbb{H}^2$.\label{t3}
        \item The components of $M\setminus \cup S$ have geometry $\hspace{0.2mm}\mathbb{H}^4$, $\mathbb{H}^3\times \mathbb{E}^1$, $\mathbb{H}^2\times\mathbb{E}^2$, or $\hspace{0.5mm}\wt{\mathbb{SL}}\times \mathbb{E}^1$.\label{t4}
        \item The components of $M\setminus \cup S$ have geometry $\mathbb{H}^2(\C)$ or $\hspace{0.5mm}\mathbb{F}^4$.\label{t5}
    \end{enumerate}
\end{theorem}

From here onwards, we use the term \emph{surface} to refer to a closed orientable $2$-manifold of genus $\ge 1$. For the facts used in the following discussion, we refer the reader to~\cite[Sections 7.1 and 7.5]{Hi}.

Due to Filipkiewicz, there are nineteen maximal $4$-dimensional geometries, one of which is not realizable by any closed manifold. If a closed manifold $M$ belongs to class~(\ref{tt1}) and
has one of the twelve aspherical geometries, then $\dcat(M) = 4$ due to Corollary~\ref{kwobv}. If $M$ belongs to class~(\ref{t4}) or~(\ref{t5}), then $M$ is aspherical and thus, $\dcat(M) = 4$. If $M$ belongs to class~(\ref{t3}) such that a geometric piece of $M$ is finitely covered by the product of a surface and a punctured surface, then $M$ is aspherical again and thus, we get $\dcat(M) = 4$. However, if a geometric piece of $M$ is finitely covered by the product of two punctured surfaces, then $M$ has only two geometric pieces and it is not aspherical. So, for new computations of $\dcat$, we are interested in these manifolds belonging to class~(\ref{t3}). 

Up to finite covers, the $4$-manifolds having $\mathbb{H}^2\times\mathbb{H}^2$ geometry are the products of two punctured surfaces with negative Euler characteristics. For $i=1,2$, let us fix some integers $g_i\ge 1$ and $k_i \ge 1$. Let $X_i$ be obtained from the surface $\Sigma_{g_i}$ after a $k_i$-puncturing of $\Sigma_{g_i}$.
If $M = X_1\times X_2$, then its double 
\[
\mathcal{D}M = M\#_{\pa M}\hspace{0.2mm}\ov{M}
\]
is a closed orientable $4$-manifold belonging to class~(\ref{t3}) described above and it is not aspherical. The simplest such example is the double of $T_o \times T_o$, where $T_o = T^2\setminus \text{int}(D^2)$ is the once-punctured $2$-torus and $D^2$ denotes the closed $2$-disc.

\begin{remark}
    The boundary $\pa M$ of $M$ is a non-trivial graph $3$-manifold. When $k_1 = k_2 = 1$, i.e., $X_1$ and $X_2$ are once-punctured surfaces, it is easy to describe $\pa M$ explicitly. Each $X_i$ has an infinite pipe, so let us visualize capping $X_i$ by a closed $2$-disc $D_i$ whose boundary $\pa D_i$ is the boundary of $X_i$ represented by a loop $t_i$.
    The boundary of $X_i\times \pa D_i$ is a $2$-torus $T_i$. Let $y_i,z_i \in H_1(T_i;\Z)$ be the canonical generators for $i = 1,2$. We choose a diffeomorphism $\phi:T_1\to T_2$ such that for the induced map $\phi_*$ in the first homology, $\phi_*(y_1)=\pm z_2$ and $\phi_*(z_1)=\pm y_2$. 
Then, 
\[
\pa M = ((X_1\times \pa D_1)\sqcup(X_2\times \pa D_2))/\sim,
\]
where $x \sim \phi(x)$ for all $x \in T_1$. Basically, $\pa M$ is formed by identifying the loops $t_1$ and $t_2$ with the commutators of $X_2$ and $X_1$, respectively. If $a_1,b_1,\ldots,a_{g_1},b_{g_1}$ are loops in $X_1$ and $c_1,d_1,\ldots,c_{g_2},d_{g_2}$ are loops in $X_2$, then $\pi_1(\pa M)$ can be described as follows.
\begin{multline*}
\pi_1(\pa M) = \bigl\{ a_1,b_1,\ldots,a_{g_1},b_{g_1},c_1,d_1,\ldots,c_{g_2},d_{g_2},t_1,t_2 \hspace{1mm}|\hspace{1mm} t_2 = \left[a_1,b_1\right]\cdots\left[a_{g_1},b_{g_1}\right], 
\\
t_1 = \left[c_1,d_1\right]\cdots\left[c_{g_2},d_{g_2}\right],\left[a_i,c_j\right]=\left[a_i,d_j\right]=\left[b_i,c_j\right]=\left[b_i,d_j\right]=1 \text{ for all } i,j \bigr\}.
\end{multline*}
From this description, it is also easy to see that $\mathcal DM$ is not aspherical. Clearly, there exists $0\ne\gamma\in\pi_1(\pa M)$ such that its image $\gamma^* \in \pi_1(M)$ under the map induced by $\pa M \hookrightarrow M$ is non-zero. Suppose $\pi_2(\mathcal{D}M)=0$. Since $\gamma^*$ is a loop in $M$ coming from $\pa M$, inside $M$, we can then attach a $2$-disc $B_1$ to $\gamma^*$ from the left side of $\mathcal{D}M$. Similarly, inside $\ov{M}$, another $2$-disc $B_2$ can be attached to $\gamma^*$ from the right side of $\mathcal{D}M$. This produces a sphere $S^2$ in $\mathcal{D}M$ inside which the loop $\gamma^*$ is contractible. This contradicts the non-triviality of $\gamma^*$. Hence, $\pi_2(\mathcal{D}M) \ne 0$.
\end{remark}

As an application of Theorem~\ref{4}, we obtain the following new computation for $\dcat$ (and hence $\cat$) of some closed non-aspherical geometrically decomposable $4$-manifolds belonging to class~(\ref{t3}) of Theorem~\ref{hillman}.

\begin{cor}\label{gap}
    If $M$ is the product of two finitely punctured closed orientable surfaces of genera $\ge 1$, then $\dcat(\mathcal{D}M) = 4$ for the double $\mathcal{D}M$ of $M$.
\end{cor}

\begin{proof}
    For $i=1,2$ and given $g_i,k_i\ge 1$, we take $M_i = \Sigma_{g_i}$, $M=X_1\times X_2$, and $Z = \Sigma_{g_1}\times\Sigma_{g_2}$ in the proof of Theorem~\ref{4} and get $\dcat(\mathcal{D}M) = 4$. 
\end{proof}

It seems that a proof of the essentiality of the closed non-aspherical geometrically decomposable $4$-manifolds belonging to class~(\ref{t3}) of Theorem~\ref{hillman} is missing from the literature. Our Corollary~\ref{gap} fills this apparent gap.

\section{For Alexandrov spaces}\label{alexalex1}

We now study A. D. Alexandrov spaces with curvature bounded below. In this section, we will extend some of our results for closed manifolds from the previous sections to a more general class of closed Alexandrov spaces, thereby initiating a study of the LS-category of Alexandrov spaces.

\subsection{Some definitions}\label{definitions}
An Alexandrov space is a finite-dimensional, complete, locally compact length space with an inner length metric that locally satisfies a lower curvature bound in the sense of distance comparison. Alexandrov spaces are locally contractible metric ANR,~\cite{Wu}. We refer to~\cite{BGP},~\cite{BBI},~\cite[Section 2]{survey}, and~\cite[Section 6]{Mit} for their precise definitions. Examples of Alexandrov spaces include orbit spaces of isometric compact Lie group actions on complete Riemannian manifolds whose sectional curvatures are uniformly bounded below. A non-manifold and non-orbifold example is the suspension of $\C P^2$, denoted $\s(\C P^2)$. For more examples and constructions, we refer the reader to~\cite[Section 2.3]{survey}.

Let $X$ be an $n$-dimensional Alexandrov space without boundary. At each point $p \in X$, one can define tangent directions and an angle distance between them, see~\cite{BGP} and~\cite{survey} for details. The metric completion of the tangent directions gives the \emph{space of directions at} $p$, denoted $\varSigma_p$, which is an Alexandrov $(n-1)$-space. If $n = 3$, then for each $p \in X$, either $\varSigma_p = S^2$ or $\varSigma_p = \R P^2$. If $n =4$, then $\varSigma_p = \s(\R P^2)$ or $\varSigma_p$ is a $3$-manifold having the spherical geometry,~\cite[Corollary 2.3]{classification}. The Euclidean cone over $\varSigma_p$ yields the \emph{tangent cone at} $p$, denoted $K_p$, which is an Alexandrov $n$-space. By Perelman's conical neighborhood theorem, every $p \in X$ has a neighborhood homeomorphic to $K_p$. A point $p \in X$ is called \emph{regular} if $\varSigma_p = S^{n-1}$; otherwise, it is called \emph{singular}. For $n \ge 3$, the codimension of the set of singular points of $X$ is at least $3$ by a theorem of Perelman. So, when $n = 3$ and $X$ is compact, the set of its singular points is finite. If all points of $X$ are regular, then $X$ is a manifold, also called an Alexandrov $n$-manifold. The converse is true when $n \le 4$ and not true otherwise~\cite[Section 2.2]{classification}.

When $n \le 2$, an Alexandrov $n$-space is a manifold, see~\cite[Corollary 10.10.3]{BBI}. So, the non-manifold cases arise when $n \ge 3$. In these cases, we first want to define a notion of the ``connected sum" of two closed Alexandrov $n$-spaces, which is a direct generalization of the notion introduced for $n = 3$ in~\cite[Section 3.1]{decomposition}. 

    Let $M$ and $N$ be two closed Alexandrov $n$-spaces. Let $p\in M$ and $q \in N$ such that $\varSigma_p = \varSigma_q$. Using Perelman's theorem, we choose neighborhoods $U \subset M$ of $p$ and $V \subset N$ of $q$ such that $U = K_p$ and $V = K_q$. Since $\pa(M\setminus U) = \varSigma_p$ and $\pa(N\setminus V) = \varSigma_q$, there exists a homeomorphism $\psi:\pa(M\setminus U) \to \pa(N\setminus V)$. 

\begin{defn}\label{alexsum}
    The \emph{connected sum of $M$ and $N$ along $(p,q)$}, denoted $M\#^{(p,q)}N$, is defined as the quotient space 
    \[
    ((M\setminus U) \sqcup (N\setminus V))/\sim,
    \]
    where $x\sim \psi(x)$ for all $x \in \pa(M\setminus U)$.
\end{defn}

Indeed, $M\#^{(p,q)}N$ is a closed Alexandrov $n$-space with curvature bounded below --- see, for example,~\cite[Section 3]{decomposition} and~\cite[Page 9]{survey}). Unlike the case of the connected sum of closed manifolds, in general, the sum $M\#^{(p_1,q_1)} N$ need not be homeomorphic to $M\#^{(p_2,q_2)} N$, even if $M = N$, $p_1=q_1$, $p_2 = q_2$, and $\varSigma_{p_1} = \varSigma_{p_2}$, as shown in~\cite[Theorem A]{decomposition} for $n = 3$. However, when $\varSigma_p = \varSigma_q = S^{n-1}$ and there is no risk of confusion regarding $p$ and $q$, we will write $M\# N$ to denote the usual connected sum of the Alexandrov spaces $M$ and $N$ along closed $n$-discs.

\subsection{Extension of some results}
To the best of the author's knowledge, the classical LS-category of closed Alexandrov spaces has not been well studied. In particular, it is not known whether the connected sum formula for closed manifolds \cite[Theorem 1]{DS} holds for the connected sum of closed Alexandrov spaces as well! Using our techniques from Section~\ref{connsums}, we provide in this section a partial such formula for $\dcat$ in the special case when one of the Alexandrov spaces is aspherical. This extends the classical connected sum formula for $\cat$ to closed non-manifold Alexandrov spaces in some cases.

The unified notion of orientability of Alexandrov spaces without boundary was studied by Mitsuishi in~\cite{Mit} (see also~\cite[Section 2]{HS}). Let $R$ be any principal ideal domain. If $X$ is a closed $R$-orientable Alexandrov $n$-space, then $X$ has a $R$-fundamental class $[X]_R\in H_n(X;R)$ by~\cite[Theorem 1.8 and Remark 2.11]{Mit}. We note that every closed Alexandrov space is $\Z_2$-orientable and $\s(\R P^2)$ is not $\Z$-orientable, see~\cite[Section 5.2]{Mit}.

\begin{theorem}[\protect{\cite[Theorem 5.1]{Mit}}]\label{mit5.1}
    Let $R$ be a principal ideal domain and $X$ be a closed $R$-orientable Alexandrov $n$-space. Then for any $R$-module $G$, the Poincaré duality homomorphism $\textup{PD}_{n}: H^n(X;G) \to H_{0}(X;G)$ is an isomorphism.
\end{theorem}

Here, the homomorphism $\text{PD}_{n}: H^n(X;G) \to H_{0}(X;G)$ sends each cohomology class $\alpha \in H^n(X;G)$ to the cap product $[X]_R\frown \alpha \in H_{0}(X;G)$.

Unlike the case of closed manifolds, the existence of fundamental class and the satisfaction of Poincaré duality is very restrictive in the case of closed Alexandrov spaces, so we will be careful while extending our results.

Let $M$ and $N$ be two closed Alexandrov $n$-spaces such that $N$ is a manifold. We can choose regular points $p\in M$ and $q \in N$ (both $M$ and $N$ have infinitely many regular points) to form $M\#^{(p,q)} N$, which we will write as $M\# N$ from here onwards. 
Let $\O_N$ be the orientation sheaf on $N$. We follow the conventions from Section~\ref{connsums}. The fundamental class $[M\#N]_R$ always exists for either $R = \Z$ or $R = \Z_2$ (a $\Z$-fundamental class exists when both $M$ and $N$ are $\Z$-orientable) and gets mapped to $[N]$ under the homomorphism induced in homology by the map $\phi:M \#N \to N$ that collapses $M$ to a point. So, Theorem~\ref{5}, which produces some new computations for the classical LS-category ($\cat$), can now be proved.

\begin{proof}[Proof sketch of Theorem~\ref{5}]
    The proof for $\dcat(M\#N) \ge n$ is the same as that of Corollary~\ref{realmain2} and uses Poincaré duality in the manifold $N$ in $\O_N \otimes \F$ coefficients, where $\F$ is a $\pi_1(N)$-module such that $\kappa_n\in H^n(N;\F)$ is the non-zero primary obstruction to a section of $\B_n(p_N):P(N)_n \to N$ due to the asphericity of $N$. Since any closed Alexandrov $n$-space $X$ is an ANR, in particular, locally contractible~\cite{Wu}, we have $\cat(X)\le n$ by~\cite[Theorem 1.7]{CLOT}. Therefore, $\dcat(M\#N) = \cat(M\#N)=n$ is obtained.
\end{proof}

\begin{ex}[\protect{Flat manifolds}]\label{extraex}
    Any flat $n$-manifold $N$ is an aspherical Alexandrov $n$-manifold. So, $\dcat(M\#N) = n$ by Theorem~\ref{5} for any Alexandrov $n$-space $M$. This gives more examples. For $n \ge 4$, let $P^{n-3}$ be an Alexandrov $(n-3)$-space. For any $t\ge 1$, the product $\#_{i=1}^t\s(\R P^2) \times P^{n-3}$ is a non-manifold Alexandrov $n$-space. So, for any flat $n$-manifold $N$, we get 
    \[
    \dcat\left(\left(\#_{i=1}^t\s(\R P^2) \times P^{n-3}\right)\#N\right) = n
    \]
    by Theorem~\ref{5}. In particular, we can take $N = \mathcal{N}^k \times T^m$ for $n=k+m$, where $T^m$ is the $m$-torus and $\mathcal{N}^k$ is the generalized Hantzsche--Wendt $k$-manifold discussed in Example~\ref{hwmani}. Furthermore, Theorem~\ref{5} also recovers the new computation $\dcat((\mathcal{N}^k\times T^m)\#(\mathcal{N}^k\times T^m)) = k+m$ obtained in Example~\ref{hwmani}.
\end{ex}

In the orientable (or $\Z$-orientable) case, a proof simpler than that of Theorem~\ref{2} and Corollary~\ref{realmain2} can be given for (possibly non-manifold) Alexandrov spaces.

\begin{proof}[Proof of Theorem~\ref{6}]
   Since $N$ is a closed orientable aspherical Alexandrov $n$-space, we can use Theorems~\ref{eg} and~\ref{mit5.1} and~\cite[Theorem 1.8]{Mit} to get the equality
   $\dcat(N) = \cd(\pi_1(N)) = n$. So, in view of Proposition~\ref{characterization}, the primary obstruction $\kappa_n \in H^n(N;\F)$ to a section of the fibration $\B_n(p_N):P(N)_n \to N$ is non-zero.
   Due to orientability, $[M\#N]$ and $[N]$ exist such that $\phi_*([M\#N]) = [N]$ for the collapsing map $\phi:M\#N\to N$. Clearly, $\phi$ induces an isomorphism of the zeroth homology groups in each coefficient system. Therefore, we have 
   \[
   \phi_*([M\#N]\frown \phi^*(\kappa_n)) = \phi_*([M\#N])\frown \kappa_n=[N] \frown \kappa_n,
   \]
   where $\phi^*(\kappa_n)\in H^n(M\#N;\phi^*(\F))$ is the primary obstruction to a lift of $\phi$ with respect to $\B_n(p_N)$.  Since $\dim(M\# N)=n$, it is the only obstruction to such a lift of $\phi$. Here, $\F$ is a $\pi_1(N)$-module, hence an abelian group. By Poincaré duality in $N$ in $\F$ coefficients (see Theorem~\ref{mit5.1}), we have $[N]\frown\kappa_n \ne 0$. This means $\phi^*(\kappa_n) \ne 0$. So, $\phi$ does not admit a lift with respect to $\B_n(p_N)$. Thus, we get
   \[
   n \le \dcat(\phi)\le \dcat(M\#N) \le \cat(M\#N) \le n
   \]
   in view of Propositions~\ref{obvious}~(\ref{t2}) and~\ref{lift}.
\end{proof}

\begin{remark}
We note that the orientability assumption is important to prove Theorem~\ref{6}. If $N$ is not orientable in the sense of~\cite{HS}, then we will be dealing with $\O_N \otimes \F$ coefficients in $N$ for the cap product $[N]\frown\kappa_n$. Here, $\O_N$ is the ``orientation bundle" on $N$, see~\cite[Section 5.3]{Mit}. So, in this case, while we can still proceed along the lines of Theorem~\ref{2}, we cannot apply Theorem~\ref{mit5.1} to get $[N]\frown \kappa_n\ne 0$ because in general, $\O_N \otimes \F \ne \F$, even if $\O_N = \Z_2$.
\end{remark}

\begin{ex}
    Following the notations of Example~\ref{extraex}, we note that Theorem~\ref{6} also recovers the new computation $\dcat((\mathcal{N}^k\times T^m)\#(\mathcal{N}^k\times T^m)) = k+m$ obtained in Example~\ref{hwmani} since $\mathcal{N}^k$ is orientable for each odd $k \ge 3$. 
\end{ex}

Theorems~\ref{5} and~\ref{6} both partially extend Corollary~\ref{realmain2} to some non-manifold cases since they include the possibility of $M$ not being a manifold. In Theorem~\ref{5}, we do not require the orientability of $N$ or $M\#N$ but need $N$ to be a manifold. In Theorem~\ref{6}, we need orientability but do not require any space to be a manifold.

Formally, Theorem~\ref{6} gives new computations of $\cat$ and $\dcat$ in orientable non-manifold cases. Unfortunately, we do not have an example of a closed aspherical Alexandrov $n$-space that is not a manifold. For $n=3$, it has been conjectured in~\cite[Conjecture 5.1]{survey} that such examples do not exist.

\section{For non-manifold Alexandrov $3$-spaces}\label{alexalex2}

We now focus on three-dimensional Alexandrov spaces that have recently been studied extensively in~\cite{classification},~\cite{GGS},~\cite{NZ}, and~\cite{decomposition}, for example.
In this section, using some well-known results from Alexandrov geometry (in particular, from the above papers), we shall make some observations regarding the classical and the distributional categories of these spaces.

\subsection{Initial observations}\label{simpobs}
Let $X$ be a closed Alexandrov $3$-space. Then we have $\dcat(X)\le \cat(X) \le 3$. If $X$ is a simply connected CW complex, then by~\cite[Theorem 1.50]{CLOT}, $\dcat(X) = \cat(X) = 1$. Hence, if $X$ is a closed positively-curved Alexandrov $3$-space that is not a manifold, then we have $\dcat(X) = \cat(X)=1$ because $X = \s(\R P^2)$ due to~\cite[Theorem 1.1]{classification} and $\s(\R P^2)$ is a simply connected CW complex.

Given a closed Alexandrov $3$-space $X$, let $M_X$ be the compact non-orientable $3$-manifold with boundary obtained by removing disjoint regular neighborhoods of the finitely many singular points of $X$. So, the boundary of $M_X$ is a finite collection of $\R P^2$. Let 
\[
M_X = M_1\#\cdots\#M_n\#l(S^2\hspace{-0.5mm}\propto\hspace{-0.5mm} S^1)
\]
be the unique prime decomposition~\cite{He} of $M_X$, where $l\ge 0$ and $S^2\hspace{-0.5mm}\propto\hspace{-0.5mm} S^1$ denotes the non-orientable $S^2$-bundle over $S^1$. For each $1\le i \le n$, let $\wh{M_i}$ be the Alexandrov $3$-space obtained from $M_i$ by capping off its $\R P^2$ boundary components. Then,
\[
X = \wh{M_1} \hspace{0.5mm}\#\cdots\#\hspace{0.5mm}\wh{M_n}\hspace{0.5mm}\#\hspace{0.5mm}l(S^2\hspace{-0.5mm}\propto\hspace{-0.5mm} S^1)
\]
is the unique (up to permutation) \emph{normal prime decomposition} of $X$ that exists due to~\cite[Theorems B and C]{decomposition}. Here, each $\wh{M_i}$ is a connected sum of prime Alexandrov $3$-spaces. The notions of prime and irreducible Alexandrov $3$-spaces and normal prime decompositions are defined in~\cite[Section 3.2]{decomposition}.

    For some $j \le n$, if $\wh{M_j}$ is orientable and aspherical, or if it is an aspherical manifold, then $\dcat(X)= 3$ in view of Theorems~\ref{5} and~\ref{6}, respectively. The latter situation is possible when $M_j$ is aspherical with no $\R P^2$ boundary components.

\begin{lemma}\label{stupid}
    Let $X$ be a closed Alexandrov $3$-space with a CW structure. If $X$ has an $S^2\hspace{-0.5mm}\propto\hspace{-0.5mm} S^1$ summand in its normal prime decomposition, then $\dcat(X)>1$.
\end{lemma}

\begin{proof}
    We use the above notations. By our assumption, $l \ge 1$. Thus, we can write $X = Y \# (S^2\hspace{-0.5mm}\propto\hspace{-0.5mm} S^1)$, where 
    \[
    Y =\wh{M_1} \#\cdots\#\wh{M_n}\#(l-1)(S^2\hspace{-0.5mm}\propto\hspace{-0.5mm} S^1).
    \]
    Since $S^2\times S^1$ is the connected orientable double cover of $S^2\hspace{-0.5mm}\propto\hspace{-0.5mm} S^1$, the space $Z = Y\#Y\#(S^2\times S^1)$ is a cover of $X$. Since $X$ is a CW complex, so is its cover $Z$. Thus, $(Z,S^2)$ is a good pair in the sense of~\cite{Ha}. So, we can proceed using the Mayer--Vietoris sequence for cohomology of the pair $(Z,S^2)$ along the lines of~\cite[Remark 8.2]{J2} to get 
    \[
    2=c\ell_{\Q}(S^2\times S^1) \le c\ell_{\Q}(Z) \le \dcat(Z).
    \]
    Finally, we use Theorem~\ref{covv} to conclude that $2 \le \dcat(Z)\le \dcat(X)$.
\end{proof}

\begin{remark}\label{difficult}
    While the classical LS-category of closed $3$-manifolds is completely known ~\cite{GLGA},~\cite{OR}, it is difficult to describe the $\cat$ of closed Alexandrov $3$-spaces completely. We list some reasons why the techniques used for determining $\cat$ of closed $3$-manifolds do not work for closed Alexandrov $3$-spaces in general.
    \vspace{-2mm}
\begin{itemize}
                    \itemsep-0.3em
    \item In the case of closed non-manifold Alexandrov $n$-spaces, Poincaré duality in middle degrees may not hold, see~\cite[Theorem 1.2 and Remark 5.2]{Mit}. If $M$ is a closed $n$-manifold for $n\ge 3$ such that $\pi_1(M)$ is not free, then $\cat(M) \ge 3$ due to~\cite[Theorem 4.1]{DKR}. Also, if $M$ is a closed $3$-manifold such that $\pi_1(M)$ is infinite and $\pi_2(M) = 0$, then $M$ is aspherical (see~\cite[Proposition 2.6]{OR}). Since the proofs of these results use Poincaré duality in middle degrees, they do not work if $M$ is a non-manifold Alexandrov space.

    \item A suitable analog of the ``projective plane theorem" for closed $3$-manifolds (see~\cite[Theorem 4.12]{He}) is not known for closed Alexandrov $3$-spaces in general. Thus, the proofs of~\cite[Lemma 3.1]{GLGA} and~\cite[Proposition 4.5]{OR} do not work for the non-manifold cases.

    \item In contrast to the $3$-manifold case, there are infinitely many closed non-irreducible prime Alexandrov $3$-spaces due to~\cite[Theorem D]{decomposition}. So, the proofs  of~\cite[Theorem 4.1]{GLGA} and~\cite[Proposition 5.1]{OR} do not work for general Alexandrov $3$-spaces $M$ to give $\cat(M) \le 2$ when $\pi_1(M)$ is free.
\end{itemize}
    \vspace{-1.5mm}
Due to these reasons, the subsequent results in~\cite{GLGA} and~\cite{OR} are not obtained, and thus, the LS-category of closed non-manifold Alexandrov $3$-spaces cannot be determined completely using just their fundamental groups.
\end{remark}

\subsection{Using curvature and cohomogeneity}\label{vsimple}

In view of Remark~\ref{difficult}, computing $\dcat$ of closed non-manifold Alexandrov $3$-spaces is even more difficult. In fact, $\dcat$ has not been determined completely even in the manifold case (see Section~\ref{threemani}). So, in this section, using some well-known results in dimension $3$, we describe up to homeomorphism closed non-manifold Alexandrov $3$-spaces $X$ for which $\cat(X)>1$ and $\dcat(X) > 1$ when $X$ is non-negatively curved or when a non-trivial compact Lie group acts nicely on $X$.

\begin{theorem}[\protect{\cite[Theorem 1.3]{classification}}]\label{curv}
    Let $X$ be a closed non-negatively curved Alexandrov $3$-space. If $X$ is not a manifold, then either
    \vspace{-2mm}
    \begin{enumerate}
                        \itemsep-0.35em
        \item $X$ is homeomorphic to $\s(\R P^2)$ or $\hspace{0.3mm}\s(\R P^2)\#\s(\R P^2)$, or
        \item $X$ is isometric to a quotient of a closed orientable flat $3$-manifold by an orientation-reversing isometric involution with only isolated fixed points.
    \end{enumerate}
\end{theorem}

Before we make our observation, we describe some flat $3$-manifolds\hspace{0.5mm}\footnote{\hspace{0.5mm}As in~\cite{flat}, we use the notations $\M_i$ and $\iota_j$ for manifolds and their involutions, respectively.}.

Let $S^1 =\{z\in\C\mid |z|=1\}$ be the unit circle and $T^2 = S^1\times S^1$ be the $2$-torus. Define $\M_2 = T^2 \times [-1,1]/\sim$, where $(z_1,z_2,1)\sim (\ov{z_1},\ov{z_2},-1)$. Then $\M_2$ is a $T^2$-bundle over $S^1$. It is a closed orientable flat $3$-manifold with holonomy group $\Z_2$. Let us define $\iota_4:\M_2\to\M_2$ as 
\[
\iota_4([z_1,z_2,t]) = [-\ov{z_1},\ov{z_2},-t].
\]
It is an orientation-reversing isometric involution having four isolated fixed points.

Let $W = S^1\times [0,1] \times [-1,1]/\sim$, where $(z,0,t)\sim(\ov{z},1,-t)$. So, $W$ is an orientable twisted interval-bundle over a Klein bottle. 
Let $W_1$ and $W_2$ be two copies of $W$, and let $g:\pa W_1 \to \pa W_2$ be the homeomorphism $g([z_1,z_2,1])=[z_2,z_1,1]$. Define $\M_6 = (W_1 \sqcup W_2)/\approx$, where $x\approx g(x)$ for all $x \in\pa W_1$. Then $\M_6$ is the closed orientable Hantzsche--Wendt $3$-manifold discussed in Example~\ref{hwmani}. Define a map $\iota_2:\M_6\to\M_6$ such that 
\[
    \iota_2([z_1,z_2,t])=\begin{cases}
    [-\ov{z_1},-z_2,t] & \text{ if } [z_1,z_2,t]\in W_1 
    \\
    [-z_1,-\ov{z_2},t] & \text{ if } [z_1,z_2,t]\in W_2. 
\end{cases}
\]
Then, $\iota_2$ is an orientation-reversing isometric involution having two isolated fixed points. 

\begin{prop}\label{simple1}
    Let $X$ be a closed non-negatively curved Alexandrov $3$-space that is not a manifold. If $\cat(X)>1$, then either $X=\M_2/\iota_4$ or $X=\M_6/\iota_2$.
\end{prop}

\begin{proof}
Since $\s(\R P^2)$ and $\s(\R P^2)\#\s(\R P^2)$ are simply connected CW complexes, their LS-category is $1$. So, by Theorem~\ref{curv}, there must exist a closed orientable flat $3$-manifold $M$ and an orientation-reversing isometric involution $\iota_M$ with isolated fixed points such that $X =M/\iota_M$. Due to Luft and Sjerve~\cite{flat}, the only possibilities are $M\in\{\M_2,\M_6,T^3\}$ and $\iota_{\M_2} = \iota_4$, $\iota_{\M_6} = \iota_2$, or $\iota_{T^3} = \iota_8$, where $\iota_8:T^3\to T^3$ defined as the involution
\[
\iota_8(z_1,z_2,z_3) =(\ov{z_1},\ov{z_2},\ov{z_3})
\]
which has eight fixed points. By~\cite[Page 9]{classification}, $T^3/\iota_8$ is a simply connected CW complex. Hence, $\cat(T^3/\iota_8)=1$ and thus, $X \ne T^3/\iota_8$. Therefore, we conclude that either $X=\M_2/\iota_4$ or $X=\M_6/\iota_2$.
\end{proof}

\begin{remark}
    We note that by~\cite[Section 2.2.1 and Lemma 5.6]{decomposition}, both $\M_2/\iota_4$ and $\M_6/\iota_2$ are not simply connected. So, our result in Proposition~\ref{simple1} is optimal in some sense. However, we do not know whether the converse statement of Proposition~\ref{simple1} is true!
\end{remark}

For a closed Alexandrov $3$-space $X$, let $\text{Iso}(X)$ denote the compact Lie group of isometries of $X$. Then the dimension of the orbit space $X/\text{Iso}(X)$ is called the \emph{cohomogeneity of the action}. We denote it by $c(X)$.  Clearly, $c(X) \in \{0,1,2,3\}$. Let $G$ be a compact Lie group acting isometrically on $X$. If $G\ne 0$ acts effectively on $X$, then $c(X)\ne 3$.

\begin{theorem}[\protect{\cite[Theorem 1.2]{NZ}}]\label{cohomo}
    Let $X$ be a closed Alexandrov $3$-space with an effective and isometric action of the circle group $S^1$. Then $X$ has $2r$ number of singular points for some $r\ge 0$ and $X$ is weakly equivariantly homeomorphic to 
    \[
    \#_{i=1}^r \s(\R P^2)\#\mathscr M,
    \]
    where $\mathscr M$ is the closed manifold of dimension $3$ described by the set of invariants $(b;(\epsilon,g,f+s,t);\{(\alpha_i,\beta_i)\}_{i=1}^n)$.
\end{theorem}

For the definitions of these invariants in terms of the action, see~\cite[Page 2]{NZ}. 

\begin{prop}\label{simple2}
Let $X$ be a closed non-manifold Alexandrov $3$-space with an effective and isometric action of a compact Lie group $G \ne 0$. If $\cat(X) > 1$, then $X = \#_{i=1}^r \s(\R P^2)\#\mathscr M$ for some $r\ge 1$ and the closed $3$-manifold $\mathscr M$ above.
\end{prop}

\begin{proof}
    Since $X$ is not a manifold, $c(X)\ne 0$. Since $G \ne 0$, the cohomogeneity cannot be $3$, i.e., $c(X)\ne 3$. So, $c(X) \in \{1,2\}$. Suppose $c(X) = 1$. Then $X = \s(\R P^2)$ due to~\cite[Theorem B]{GGS} and hence, $\cat(X) = 1$ because $\s(\R P^2)$ is simply connected. 
    This contradicts our hypothesis. Therefore, we must have $c(X)= 2$. This means the orbits are $1$-dimensional. Hence, $G = S^1$ by definition. So, we can use Theorem~\ref{cohomo} to get $X = \#_{i=1}^r \s(\R P^2)\#\mathscr M$ for some $r\ge 1$.
\end{proof}

\begin{remark}
    The statements of Propositions~\ref{simple1} and~\ref{simple2} hold true if the condition $\cat(X)>1$ is replaced by the condition $\dcat(X)>1$ since $\cat(X)\ge \dcat(X)$.
\end{remark}

We end this paper by giving two instances where the converse of Proposition~\ref{simple2} holds true.

\begin{ex}
    Let $X = \#_{i=1}^r \s(\R P^2)\#\mathscr M$ for some $r\ge 1$ as in Theorem~\ref{cohomo}. If $\mathscr M$ is aspherical, then we have $\dcat(X) = \cat(X) = 3$ by Theorem~\ref{5}. Next, let 
    \[
    X = \s(\R P^2)\#\s(\R P^2)\#(S^2\times S^1)
    \]
    with any one of the two effective and isometric $S^1$ actions described explicitly in~\cite[Example 5.3]{NZ}. It follows from the technique of the proof of Lemma~\ref{stupid} that $2 = c\ell_{\Q}(S^2\times S^1) \le \dcat(X) \le \cat(X)$.
\end{ex}

\section*{Acknowledgement}
The author would like to thank Alexander Dranishnikov for various insightful discussions related to this project. The author also thanks Luca Di Cerbo for helpful discussions about four-dimensional manifolds and symplectically aspherical manifolds, and John Oprea for discussions related to the formalism of the cap property. The author is grateful to the anonymous referee for several comments and suggestions that helped improve the exposition of this paper.

\end{document}